\documentclass[draft]{birkjour}
\usepackage[noadjust]{cite}
\usepackage{xcolor}
\usepackage{mathrsfs}
\RequirePackage[all]{xy}

\newtheorem{theorem}{Theorem}[section]
\newtheorem{corollary}[theorem]{Corollary}

\newtheorem{proposition}[theorem]{Proposition}
\theoremstyle{definition}
\newtheorem{definition}[theorem]{Definition}
\theoremstyle{remark}
\newtheorem*{remark}{Remark}

\numberwithin{equation}{section}

\newcommand{\BibTeX}{B\kern-0.1emi\kern-0.017emb\kern-0.15em\TeX}
\newcommand{\XYpic}{$\mathrm{X\kern-0.3em\raisebox{-0.18em}{Y}}$-$\mathrm{pic}\,$}

\newcommand{\cl}{C \kern -0.1em \ell}  

\newcommand{\bx}{\boldsymbol{x}}
\newcommand{\bq}{\boldsymbol{q}}
\newcommand{\ubx}{\underline{\boldsymbol{x}}}
\newcommand{\ubq}{\underline{\boldsymbol{q}}}
\newcommand{\be}{\begin{eqnarray*}}
	\newcommand{\ee}{\end{eqnarray*}}

\allowdisplaybreaks

\newcommand{\ed}{\end{document}}
\begin{document}

%
%
%
%
%
%
%
%
%

\title[ Generalized $\Pi$-operator in the theory of slice monogenic functions]
 {Generalized $\Pi$-operator in the theory of slice monogenic functions and applications }
\author[Ziyi Sun]{Ziyi Sun}
\address{%
School of Mathematical Sciences,\\ Anhui University, Hefei, P.R. China}
\email{A23201027@stu.ahu.edu.cn}

\author[Chao Ding]{Chao Ding}
%
\address{%
Center for Pure Mathematics, \\School of Mathematical Sciences,\\ Anhui University, Hefei, P.R. China}
\email{cding@ahu.edu.cn}

%
\date{\today}
\begin{abstract}
The $\Pi$-operator plays an important role in complex analysis, especially in the theory of generalized analytic functions in the sense of Vekua. In this paper, we introduce a generalized $\Pi$-operator in the theory of slice monogenic functions, and some mapping properties of the generalized $\Pi$-operator are also introduced. Further, a left and right inverse and the adjoint operator of the generalized $\Pi$-operator are given. As an application, we introduce a slice Beltrami equation, which reduces to the classical complex Beltrami equation when the dimension is $2$. We show details that the norm estimate of the generalized $\Pi$-operator can determine the existence of solutions of the slice Beltrami equation.
\end{abstract}
\label{page:firstblob}
\maketitle
\section{Introduction}\hspace*{\fill} 
Classical Clifford analysis is considered as a generalization of complex analysis in higher dimensions. At the heart of the theory is the study of  Dirac operator, which is a generalization of the Cauchy-Riemann operator in higher dimensions, and the null solutions to the Dirac operator are called monogenic functions. The theory of monogenic functions has been fully developed in the last decades, which preserves most properties of holomorphic functions, for instance, Liouville's Theorem, the Cauchy integral formula, the mean value property, and the Cauchy Theorem, etc. For more details on classical Clifford analysis, see \cite{2,10}.
\par 
	It is well-known that polynomials given in terms of the complex variable are known to be holomorphic in complex analysis. However, due to the non-commutativity of multiplications of Clifford numbers, polynomials given in terms of vectors in higher dimensions are no longer monogenic. This was changed in 2006 when the concept of slice regular functions on quaternions was introduced by Gentili and Struppa \cite{Ge1,Ge2}, which was motivated by Cullen's earlier research \cite{Cu}. Then, in 2010, Colombo, Sabadini and Struppa \cite{Co2} generalized this idea to the general Clifford algebras with the concept of slice monogenic functions. Slice monogenic functions were firstly defined as functions, which are holomorphic on each slice of the Euclidean space. Later on, many contributions were made to the study of slice monogenic functions. For instance, in \cite{Co4}, the authors introduced a Cauchy integral formula stating that the value of a slice monogenic function at some point $p$ in the domain can be represented by an integral over the boundary with a monogenic kernel. It is worth pointing out that the point $p$ is not necessarily in the slice. The theory of slice monogenic functions has been well-developed so far, see, e.g. \cite{Co1}.
\par 
In 2011, Ghiloni and Perotti  \cite{Gh} developed the theory of slice regular functions on real alternative algebras with the concepts of stem functions and slice functions. This method allows further development of the slice regular function theory.  Bisi and Winkelmann developed a mean value formula for slice regular functions in \cite{Bi}, which leads to a crucial conclusion stating that a slice regular function over quaternions is also harmonic in a specific sense.  The volume  Borel-Pompeiu formula and the volume Cauchy integral formula for slice regular functions were firstly presented by Ghiloni and Perotti in \cite{Gh1}. More details regarding slice regular/monogenic functions can be found in \cite{Co1,Gh,Gh1,Ge3,Co3}. 
\par 
In 2013, a non-constant coefficients differential operator \emph{G} was introduced by Colombo et al. in \cite{Co4}. It turns out that null solutions to $G$ are strongly connected to slice monogenic functions, when the domains are given with appropriate constraints. Studying this differential operator $G$ has the benefit of providing an explicit global differential operator for defining slice monogenic functions. Later on, many researchers started to
study this global differential operator for slice mongenicity. For instance, in \cite{Gh1},
a volume Cauchy integral formula and a volume Borel Pompeiu formula for slice regular functions on real associative *-algebras were introduced by Ghiloni and Perotti. In \cite{22}, the authors discovered a global Borel-Pompeiu formula and a global Cauchy-type formula for a modified non-constant coefficients differential operator were presented for quaternionic slice regular functions. In \cite{Ding}, the authors prove that the Teodorescu transform is also the right inverse of the slice Cauchy-Riemann operator and give some other properties of the Teodorescu transform. It is well-known that one of the applications of the Teodorescu transform 
 is to study existence of solutions of the Beltrami equation.
\par
Beltrami equation, as a generalization of Cauchy-Riemann's equation, has a wide range of applications in the fields of hydrodynamics, electrodynamics and modern control theory. Formally, Beltrami equation can be divided into the following two types: the first type $f_{\bar{z}}(z)=\mu(z)f_z(z)$ in which $\mu(z)$ is a measurable function; the second type $\overline{{f_{\bar{z}}}}(z)=a(z)f_z (z)$ in which $a(z)$ is an analytic function; usually $\mu(z)$ and $a(z)$ are called the complex characteristic of $f$. For a long time, the problem of existence and uniqueness of homogeneous solutions of Beltrami equation has been one of the hot topics concerned by many mathematicians. Back in 1996, G\"{u}rlebeck and K\"{a}hler in \cite{G2} dealt with a hyper-complex generalization of  the complex  $\Pi$-operator which turns out to have most of the useful properties of its complex origin and presented an application of the generalized $\Pi$-operator to the solution of a hypercomplex Beltrami equation, and called the equation $D\omega=\boldsymbol{q}\overline{D}\omega$ \emph{generalized Beltrami  equation}. In 2016, Abreu Blaya  et al. mentioned the application of the generalized $\Pi$-operator to solve the Beltrami equation in \cite{Ri}. More details on $\Pi$-operator can be found in the books \cite{G1, Vekua}.
\par
 In this paper we will give that the definition of a generalized $ \Pi$-operator in the theory of slice monogenic functions with the Teodorescu transform studied in \cite{Ding}. Further, we will provide an integral representation formula, continuity, norm estimations, and some algebraic properties for the generalized $\Pi$-operator.  
\par 
This paper is organized as follows. Some preliminaries on slice monogenic functions are introduced in Section 2. In Section 3, We introduce the generalized $\Pi$-operator and its mapping properties in the theory of slice monogenic functions. In particular, a left and right inverse and the adjoint operator of the generalized $\Pi$-operator are given. In Section 4, the norm estimate of the generalized $\Pi$-operator is applied to determine existence of solutions of a slice Beltrami equation as an application. 
\section{Preliminaries}
In this section, we review some  definitions and results on slice Clifford analysis. More details can be found in \cite{Co1}.
\par
Let $\mathbb{R}^{m}$ be the $m$-dimensional Euclidean space with a standard orthonormal basis $\left \{\boldsymbol{e}_{1}, . . . , \boldsymbol{e}_{m}\right \}$. The real Clifford algebra $\mathcal{C}l_{m}$ is generated by $\mathbb{R}^{m}$ with the relationship
$$\boldsymbol{e}_{i}\boldsymbol{e}_{j}+\boldsymbol{e}_{j}\boldsymbol{e}_{i}=-2\delta_{ij},$$
where $\delta_{ij}$ is the Kronecker delta function. Hence a Clifford number $ \boldsymbol{x} \in \mathcal{C}l_{m}$ can be written as $ \boldsymbol{x} =\sum_A  x_A\boldsymbol{e}_A$ with real coefficients and  $A\subset\left \{1, . . . , m\right \}$. We introduce a norm for a Clifford number $\boldsymbol{x} =\sum_{A}  x_{A}\boldsymbol{e}_{A}$  as $\lvert \boldsymbol{x}\rvert$ =$(\sum_{A}  x_{A}x^{2}_{A})^{\frac{1}{2}}$. If we denote $\mathcal{C}l_{m}^{k}=\left \{ \boldsymbol{x} \in \mathcal{C}l_{m}:\boldsymbol{x}=\sum_{\lvert A\rvert=k}x_{A}\boldsymbol{e}_{A}\right \}$, where $ \lvert A\rvert $ stands for the cardinality of the set $A$, then one can see that $\mathcal{C}l_{m}$=$\oplus_{k=0}^{m}$$\mathcal{C}l_{m}^{k}$. In particular, the $(m+1)$-dimensional
Euclidean space $\mathbb{R}^{m+1}=\mathbb{R}\oplus\mathbb{R}^m$ can be identified with $\mathcal{C}l_{m}^{0}$ $\oplus$ $\mathcal{C}l_{m}^{1}$ as the following $$\mathbb{R}^{m+1}\rightarrow \mathcal{C}l_{m}^{0} \oplus \mathcal{C}l_{m}^{1},$$ $$(x_{0},x_{1},...,x_{m})\mapsto x_{0}+x_{1}\boldsymbol{e}_{1}+...+x_{m}\boldsymbol{e}_{m}.$$
Further, we call elements in $\mathcal{C}l_{m}^{1}$ vectors and elements in $\mathcal{C}l_{m}^{0}\oplus\mathcal{C}l_{m}^{1}$ paravectors. For an arbitrary Clifford number $\boldsymbol{x} =\sum_{\lvert A\rvert=k}x_{A}\boldsymbol{e}_{A}$, we define the Clifford conjugation of $\boldsymbol{x}$ by
$$\overline{\boldsymbol{x}}=\sum_{A}(-1)^{\frac{\lvert A\rvert(\lvert A \rvert+1)}{2}}x_{A}\boldsymbol{e}_{A}. $$
\par
We denote a vector by $\underline{\boldsymbol{x}}=\sum_{k=1}^{m}x_{k}\boldsymbol{e}_{k}$. Then given a paravector $\boldsymbol{x}$, if $\boldsymbol{x}\notin\mathbb{R}$, we can write it as $\boldsymbol{x} = x_{0} +\underline{\boldsymbol{x}} =: \text{Re}[\boldsymbol{x}] + \dfrac{\underline{\boldsymbol{x}}}{\lvert\underline{\boldsymbol{x}}\rvert}\lvert\underline{\boldsymbol{x}}\rvert=:u+I_{\boldsymbol{x}}v$, where $u= x_{0}, v=\lvert\underline{\boldsymbol{x}}\rvert=(\sum_{j=1}^{m} x_{j}^{2})^{\frac{1}{2}}$ and $I_{\boldsymbol{x}}=\dfrac{\underline{\boldsymbol{x}}}{\lvert\underline{\boldsymbol{x}}\rvert}$;
if $\boldsymbol{x} \in \mathbb{R}$, which means $v = 0$, then we assume that $I_{\boldsymbol{x}}$ is an arbitrary unit vector in $\mathbb{S}$, where $\mathbb{S}$ stands for the set of unit vectors defined by $$\mathbb{S}:=\left \{ \underline{\boldsymbol{x}}=\boldsymbol{e}_{1}x_{1}+...+\boldsymbol{e}_{m}x_{m}\in \mathcal{C}l_{m}^{1}:x_{1}^{2}+...+x_{m}^{2}=1 \right \}.$$
One can easily observe that $\mathbb{S}$ is the unit sphere in $\mathbb{R}^{m}$ and if ${I} \in \mathbb{S}$, then $\textit {I}^{2}=-1$.
Further, for ${I} \in \mathbb{S}$, let $\mathbb{C}_{I}$ be the plane generated by 1 and ${I}$, which is isomorphic to the complex plane. An element in $\mathbb{C}_{I}$ will be denoted by $u + Iv$ with $u, v \in  \mathbb{R}$.
Hence, given a paravector $\boldsymbol{x}$, one can rewrite it as an element in a suitable complex plane $\mathbb{C}_{I}$. Let $\boldsymbol{s} = s_{0} + \underline{\boldsymbol{s}} = s_{0} + I_{s}\lvert\underline{\boldsymbol{s}}\rvert \in  \mathbb{R}^{m+1}$, we denote by $\left[\boldsymbol{s}\right]$ the set
$$\left[\boldsymbol{s}\right] = \left\{\boldsymbol{x} \in  \mathcal{C}l_{m}^{1}: \boldsymbol{x} = s_{0} + I\lvert\underline{\boldsymbol{s}}\rvert, I \in \mathbb{S}\right\}.$$
One can observe that the set $\left[\boldsymbol{s}\right]$ is either reduced to a point (when $\boldsymbol{s} \in \mathbb{R} $) or it is the $(m-1)$-sphere with center at $ s_{0}$ and radius $\lvert\underline{\boldsymbol{s}}\rvert$.
\begin{definition}
	Let $ \Omega\subset \mathbb{R}_*^{m+1} $ be a domain and $f :\Omega \longrightarrow \mathcal{C}l_{m}$ be a real differentiable function. Let $I \in \mathbb{S}, f_I $ is the restriction of $f$ to the complex plane $\mathbb{C}_{I}$ and $u + Iv$ is an element in $\mathbb{C}_{I}$. f is called \emph{left slice monogenic}, if for all  $I \in \mathbb{S}$ we have $\dfrac{1}{2}\left(\dfrac{\partial}{\partial u}+I\dfrac{\partial}{\partial v}\right)f_{I}(u + Iv)=0 $ on $ \Omega_{I}:=\Omega\cap\mathbb{C}_{I}$.
\end{definition}
As analogs of $ \partial$ and $\bar{\partial}$ in one dimensional complex analysis, we define the notion of {I}-derivative as $\partial_{I}:=\dfrac{1}{2}\left(\dfrac{\partial}{\partial u}-I\dfrac{\partial}{\partial v}\right),\ \bar{\partial}_{I}:=\dfrac{1}{2}\left(\dfrac{\partial}{\partial u}+I\dfrac{\partial}{\partial v}\right)$. Later in this paper, we also use $\partial_{\boldsymbol{x}_{I}}$ and $\partial_{\boldsymbol{q}_{I}}$ to specify the variable that the differential operator depends on if necessary. Now, let us recall some definitions as follows. 
 \begin{definition}
 	Given a set $D \subset \mathbb{C},$ which is invariant with respect to complex conjugation, a function $F : D \longrightarrow \mathcal{C}l_m \otimes\mathbb{C}$ satisfying $ F(\bar{z}) =\overline{ F(z)}$ for all $z \in  D$ is called a \emph{stem function} on $D$.
 \end{definition}
 Let $J \in  \mathbb{S}$ and $\Phi_{J} : \mathbb{C} \longrightarrow\mathbb{C_{\textit{J}}}$ be the canonical isomorphism which maps $u + iv$ to $u + Jv$. Given an open set $D \subset \mathbb{C}$, we denote
 $$\Omega_{D}=\bigcup_{J\in \mathbb{S}}\Phi_{J}(D)\subset\mathbb{R}^{m+1}.$$
 If an open set $ \Omega \subset\mathbb{R}_*^{m+1}$ satisfies $ \Omega = \Omega_{D}$, then we say that it is \emph{axially symmetric}. According to the definition of a stem function, such a function $F =
 F_{1} + iF_{2}$ on $D$ with $F_{1},F_{2} : D \longrightarrow \mathcal{C}l_{m}$ and $z = u + iv$ induces the \emph{slice function} $f = \mathcal{I}(F): \Omega_{D} \longrightarrow \mathcal{C}l_{m}$, which satisfies 
 \begin{align*}
 f(\boldsymbol{x})= F_{1}(z)+JF_{2}(z),\ \text{if}\ \boldsymbol{x} = u+Jv = \Phi_{J}(z) \in \Omega_{D}\cap\mathbb{C_{\textit{J}}}.
 \end{align*}
 We denote the set of (left) slice functions on $\Omega_{D}$ by
 \begin{align*}
 \mathcal{S}(\Omega_{D}):= \left\{ f:\Omega_{D} \longrightarrow \mathcal{C}l_{m} \mid f = \mathcal{I}(F), F:D \longrightarrow \mathcal{C}l_{m}\otimes \mathbb{C}\ \text{stem function}\right\} .
  \end{align*}
An important property of the slice function is the following representation formula.
\begin{theorem}[Representation formula]\cite{Gh} Let $\Omega_D \subset\mathbb{R}_*^{m+1}$ be a bounded axially symmetric domain. Further, let $f : \Omega_{D} \longrightarrow\mathcal{C}l_{m}$ be a slice function. Then, for any $I \in \mathbb{S} $ and $\boldsymbol{x} = u + I_{\boldsymbol{x}}v \in \Omega_{D}$, where  $I_{\boldsymbol{x}} \in \mathbb{S}$, we have  \begin{align*}
 f(\boldsymbol{x})=\dfrac{1-I_{\boldsymbol{x}}I}{2}f(u+Iv)+\dfrac{1+I_{\boldsymbol{x}}I}{2}f(u-Iv).
  \end{align*}
\end{theorem}
The definition of slice monogenic functions induced by stem functions is as follows.
 \begin{definition}
 	 Let $D$ be a symmetric domain in $\mathbb{C}$ and $\Omega_D \subset \mathbb{R}^{m+1}$ be defined as above. Then, a slice function $f :\Omega_D \longrightarrow Cl_m $ is called \emph{slice monogenic} if its stem function $F = F_1 +iF_2 : D \longrightarrow Cl_m \otimes \mathbb{C}$ is holomorphic, in other words,
 	its components $F_1, F_2$ satisfy the Cauchy-Riemann equations:
 \begin{align*}
 \frac{\partial F_1}{\partial u}=\frac{\partial F_2}{\partial v},\ \frac{\partial F_1}{\partial v}=-\frac{\partial F_2}{\partial u},\ z=u+iv\in  D.
 	\end{align*}
 \end{definition}
 Recall that the \emph{slice Cauchy-Riemann operator} is given by 
\begin{align*}G=\dfrac{\partial}{\partial x_{0}} +\dfrac{\underline{\boldsymbol{x}}}{\lvert\underline{\boldsymbol{x}}\rvert^{2}}\sum_{j=1}^{m}{x_{j}\dfrac{\partial}{\partial{x_{j}}}}:=\dfrac{\partial}{\partial x_{0}} +\dfrac{\underline{\boldsymbol{x}}}{\lvert\underline{\boldsymbol{x}}\rvert^{2}}E_{\ubx}, \end{align*}
where one can see that $E_{\ubx}=\sum_{j=1}^{m}{x_{j}\dfrac{\partial}{\partial{x_{j}}}}$ is the Euler operator. This differential operator is slightly different from the operator given in \cite{Co4} by a factor $\lvert\underline{\boldsymbol{x}}\rvert^{2} $. Since $\lvert\underline{\boldsymbol{x}}\rvert^{2} $ creates singularities, we should keep $\boldsymbol{x}$ off the real line to ensure that $\lvert\underline{\boldsymbol{x}}\rvert^{2}\neq 0 $. 
Hence, we introduce the notation $\mathbb{R}_*^{m+1}:= \mathbb{R}^{m+1}\setminus\mathbb{R}$ for the rest of this article.
\begin{remark}
The space of slice monogenic functions coincides with the kernel space of the slice Cauchy-Riemann operator under certain conditions on the domain, see \cite{Co4}.
\end{remark}
\begin{definition}
	Let $f,g\in C^1(\overline{\Omega_{I}})$ be monogenic functions and
	\begin{align*}
		f(\boldsymbol{q})=\sum_{n\in \mathbb{N}}\boldsymbol{q}^n a_n,\ 
	g(\boldsymbol{q})=\sum_{n\in \mathbb{N}}\boldsymbol{q}^n b_n
	\end{align*}
	 be their power series expansions. The $\ast$-product of $f$ and $g$ is the monogenic function defined by 
	$$f\ast g=\sum_{n\in \mathbb{N}}\boldsymbol{q}^n \sum_{k=0}^{n} a_k b_{n-k}  .$$
\end{definition}
 Recall  that the \emph{slice Cauchy kernel} for slice monogenic functions is denoted by \begin{align*} S^{-1}(\boldsymbol{q},\boldsymbol{x})=-(\boldsymbol{q}^{2}-2Re[\boldsymbol{x}]\boldsymbol{q}+\lvert{\boldsymbol{x}}\rvert^{2})^{-1}(\boldsymbol{q}-\overline{\boldsymbol{x}}),\end{align*} where $\boldsymbol{x},\boldsymbol{q}\in\mathbb{R}^{m+1}$.
\par
The Cauchy kernel for the slice Cauchy-Riemann operator $G$ is given by \begin{align*}K (\boldsymbol{q},\boldsymbol{x})=\dfrac{2S^{-1}(\boldsymbol{q},\boldsymbol{x})}{\omega_{m-1}\lvert\underline{\boldsymbol{x}}\rvert^{m-1}}, \end{align*}
where $\omega_{m-1}$ is the area of the $(m-1)$-sphere $\mathbb{S}$.
\par
Now, we introduce two integral operators as follows.
\begin{align*}
&T_{\Omega_{D}}f(\boldsymbol{q})=-\frac{1}{2\pi}\int_{\Omega_{D}}K(\boldsymbol{q},\boldsymbol{x})f(\boldsymbol{x})dV(\boldsymbol{x}),\\
&F_{\partial\Omega_{D}}f(\boldsymbol{q})=-\frac{1}{2\pi}\int_{\partial\Omega_{D}}K(\boldsymbol{q},\boldsymbol{x})n(\boldsymbol{x})f(\boldsymbol{x})d\sigma(\boldsymbol{x}),
\end{align*}
 where $n( \boldsymbol{x})$ is the outward unit normal vector to the boundary $\partial\Omega_{D}$, $d\sigma(\boldsymbol{x})$  is the area
element on $\partial\Omega_{D}$ and $dV(\boldsymbol{x})$ is the volume element in $\Omega_{D}$. The operator $T_{\Omega_{D}}$ is usually called the \emph{Teodorescu transform}. Hence, with these notations, the Borel-Pompeiu formula can be rewritten as 
\begin{align}\label{BPF1}
F_{\partial\Omega_D} f(\boldsymbol{q})+T_{\Omega_D}(Gf)(\boldsymbol{q})=f(\boldsymbol{q}).
\end{align}


\section{Definitions and properties of the generalized $\Pi $-operator}
In the theory of slice monogenic functions, the definition of the generalized $\Pi $-operator is given as follows.
\begin{definition}
	Let $\Omega_D$ be an axially symmetric domain in $\mathbb{R}_*^{m+1}$ and $f\in C^1 (\overline{\Omega_D})$. The generalized $\Pi$-operator in the theory of slice monogenic functions is denoted by
	\begin{align*}
	\Pi_{\Omega_{D}} (\boldsymbol{q}) =\overline{G_{\boldsymbol{q}}}T_{\Omega_{D}} (\boldsymbol{q}) ,\ \boldsymbol{q}\in\Omega_{D},
	\end{align*}
	where 
	\begin{align*}
	\overline{G_{\bq}}=\dfrac{\partial}{\partial q_{0}} -\dfrac{\underline{\boldsymbol{q}}}{\lvert\underline{\boldsymbol{q}}\rvert^{2}}\sum_{j=1}^{m}{q_{j}\dfrac{\partial}{\partial{q_{j}}}}.
	\end{align*}
\end{definition}
To describe properties of $\Pi_{\Omega_D}$ in the later sections, we also need a slice $\Pi$-operator given by
$$ \Pi_{\Omega_I}f(\boldsymbol{q}):=\overline{G_{\boldsymbol{q}}}T_{\Omega_I}f(\boldsymbol{q}),$$ 
where $$ T_{\Omega_I}f(\boldsymbol{q}):=-\frac{1}{2\pi}\int_{\Omega_{I}}S^{-1}(\boldsymbol{q},\boldsymbol{x})f(\boldsymbol{x})dV_I(\boldsymbol{x}).$$
Next, we apply the differential operator $\overline{G}$ to $T_{\Omega_{I}}$ to get an integral representation formula for the $\Pi_{\Omega_{I}}$-operator as follows. 
\begin{theorem}\label{Thm3.2}
	Let $\Omega_D \subset\mathbb{R}_*^{m+1}$ be a bounded axially symmetric domain. Assume that  $f\in C^1(\overline{\Omega_{I}})$ with $I\in\mathbb{S}$. Then
	\begin{align*}\Pi_{\Omega_{I}}f(\boldsymbol{q})=-\frac{1}{\pi}\int_{\Omega_{I}} (\boldsymbol{x}-\boldsymbol{q})^{-*2}f(\boldsymbol{x})dV_{I}(\boldsymbol{x}).
	\end{align*}
\end{theorem}
\begin{proof}
We denote 
\begin{align*}
&\boldsymbol{q}=q_{0}+\underline{\boldsymbol{q}}=q_{0}+I_{\boldsymbol{q}}\lvert\underline{\boldsymbol{q}}\rvert=:q_{0}+{I}_{\boldsymbol{q}}\zeta,\\
&\boldsymbol{x}=x_0+\underline{\boldsymbol{x}}=x_{0}+I\lvert\underline{\boldsymbol{x}}\rvert=:x_{0}+I\eta,
\end{align*}
 where $ {I}_{\boldsymbol{q}}=\dfrac{\underline{\boldsymbol{q}}}{\lvert\underline{\boldsymbol{q}}\rvert}, \ I=\dfrac{\underline{\boldsymbol{x}}}{\lvert\underline{\boldsymbol{x}}\rvert}$, $\eta=\lvert\underline{\boldsymbol{x}}\rvert$ and 
$ \zeta=\lvert\underline{\boldsymbol{q}}\rvert$.  Let $\boldsymbol{q}_{\pm I}=q_{0}\pm I\zeta$, $\overline{\partial_{\boldsymbol{q}_{I}}}=\frac{1}{2}(\partial_{q_{0}}+I\partial_{\zeta})$ and $\partial_{\boldsymbol{q}_{I}}=\frac{1}{2}(\partial_{q_{0}}-I\partial_{\zeta})$.
Using \cite[Theorm 3.4]{Ding}, we obtain that, for any $f\in C^{1}(\overline{\Omega_{I}})$ and $\boldsymbol{q}\in\Omega_{I}$,
\begin{align*}
	2\pi\dfrac{\partial}{\partial q_{0}} T_{\Omega_{I}}f(\boldsymbol{q})=-\int_{\Omega_{I}}\dfrac{\partial}{\partial q_{0}} S^{-1}(\boldsymbol{q},\boldsymbol{x})f(\boldsymbol{x})dV_{I}(\boldsymbol{x})+\pi f(\boldsymbol{q})
\end{align*}
and
\begin{align*}
	2\pi\dfrac{\boldsymbol{q}}{\lvert\underline{\boldsymbol{q}}\rvert^{2}}\sum_{i=1}^{m}q_{i}\dfrac{\partial}{\partial q_{i}} T_{\Omega_{I}}f(\boldsymbol{q})=-\int _{\Omega_{I}}I_{\boldsymbol{q}}\dfrac{\partial}{\partial \zeta}S^{-1}(\boldsymbol{q},\boldsymbol{x})f(\boldsymbol{x})dV_{I}(\boldsymbol{x})+\pi f(\boldsymbol{q}).
	\end{align*} 
Then,
\begin{align*}
2\pi\overline{G_{\boldsymbol{q}}}T_{\Omega_{I}}f(\boldsymbol{q})&=2\pi\dfrac{\partial}{\partial q_{0}} T_{\Omega_{I}}f(\boldsymbol{q})-2\pi\dfrac{\boldsymbol{q}}{\lvert\underline{\boldsymbol{q}}\rvert^{2}}\sum_{i=1}^{m}q_{i}\dfrac{\partial}{\partial q_{i}} T_{\Omega_{I}}f(\boldsymbol{q})\\
&=-\int _{\Omega_{I}}\bigg(\dfrac{\partial}{\partial q_{0}}-I_{\boldsymbol{q}}\dfrac{\partial}{\partial \zeta}\bigg)S^{-1}(\boldsymbol{q},\boldsymbol{x})f(\boldsymbol{x})dV_{I}(\boldsymbol{x}).
\end{align*}
Now, we denote 
 \begin{align*}a&=\boldsymbol{q}^{2}-2Re[\boldsymbol{x}]\boldsymbol{q}+\lvert{\boldsymbol{x}}\rvert^{2}=(q_0+I_{\boldsymbol{q}}\zeta)^2-2x_0(q_0+I_{\boldsymbol{q}}\zeta)+\lvert \boldsymbol{x}\rvert^{2},\\
 	b&=\boldsymbol{q}-\overline{\boldsymbol{x}}=(q_0+I_{\boldsymbol{q}}\zeta)-\overline{\boldsymbol{x}}.
 \end{align*}
 Then, we have
\begin{align*}
&\dfrac{\partial}{\partial q_{0}}S^{-1}(\boldsymbol{q},\boldsymbol{x})=-\dfrac{\partial}{\partial q_{0}}(\boldsymbol{q}^{2}-2Re[\boldsymbol{x}]\boldsymbol{q}+\lvert{\boldsymbol{x}}\rvert^{2})^{-1}(\boldsymbol{q}-\overline{\boldsymbol{x}})	\\
&=-\dfrac{\partial}{\partial q_{0}}\left[(q_0+I_{\boldsymbol{q}}\zeta)^2-2x_0(q_0+I_{\boldsymbol{q}}\zeta)+\lvert \boldsymbol{x}\rvert^{2}\right]^{-1}\left[(q_0+I_{\boldsymbol{q}}\zeta)-\overline{\boldsymbol{x}}\right]\\
&=-\left[-a^{-2}(2(q_0+I_{\boldsymbol{q}}\zeta)-2x_0)b+a^{-1}\right]\\
&=-a^{-2}\left[a-2(\boldsymbol{q}-x_0)b\right],
\end{align*}
and
\begin{align*}
&I_{\boldsymbol{q}}\dfrac{\partial}{\partial \zeta}S^{-1}(\boldsymbol{q},\boldsymbol{x})=-I_{\boldsymbol{q}}\dfrac{\partial}{\partial \zeta}(\boldsymbol{q}^{2}-2Re[\boldsymbol{x}]\boldsymbol{q}+\lvert{\boldsymbol{x}}\rvert^{2})^{-1}(\boldsymbol{q}-\overline{\boldsymbol{x}})\\
&=-I_{\boldsymbol{q}}\dfrac{\partial}{\partial \zeta}\left[(q_0+I_{\boldsymbol{q}}\zeta)^2-2x_0(q_0+I_{\boldsymbol{q}}\zeta)+\lvert \boldsymbol{x}\rvert^{2}\right]^{-1}\left[(q_0+I_{\boldsymbol{q}}\zeta)-\overline{\boldsymbol{x}}\right]\\
&=-I_{\boldsymbol{q}}\left[-a^{-2}(2(q_0+I_{\boldsymbol{q}}\zeta)I_{\boldsymbol{q}}-2x_0I_{\boldsymbol{q}})b+a^{-1}I_{\boldsymbol{q}}\right]\\
&=-I_{\boldsymbol{q}}\left[-a^{-2}2(\boldsymbol{q}-x_0)b+a^{-1}\right]I_{\boldsymbol{q}}\\
&=a^{-2}\left[a-2(\boldsymbol{q}-x_0)b\right].
\end{align*}
Hence, we obtain
\begin{align*}
&\bigg(\dfrac{\partial}{\partial q_{0}}-I_{\boldsymbol{q}}\dfrac{\partial}{\partial \zeta}\bigg)S^{-1}(\boldsymbol{q},\boldsymbol{x})\\
=&-a^{-2}\left[a-2(\boldsymbol{q}-x_0)b\right]-a^{-2}\left[a-2(\boldsymbol{q}-x_0)b\right]\\
=&-2a^{-2}\left[a-2(\boldsymbol{q}-x_0)b\right],
\end{align*}
and 
\begin{align*}
&2\pi\overline{G_{\boldsymbol{q}}}T_{\Omega_{I}}f(\boldsymbol{q})
=-\int _{\Omega_{I}}\bigg(\dfrac{\partial}{\partial q_{0}}-I_{\boldsymbol{q}}\dfrac{\partial}{\partial \zeta}\bigg)S^{-1}(\boldsymbol{q},\boldsymbol{x})f(\boldsymbol{x})dV_{I}(\boldsymbol{x})\\
&=\int _{\Omega_{I}}2a^{-2}\left[a-2(\boldsymbol{q}-x_0)b\right]f(\boldsymbol{x})dV_{I}(\boldsymbol{x})\\
&=\int _{\Omega_{I}}2\left[\boldsymbol{q}^{2}-2Re[\boldsymbol{x}]\boldsymbol{q}+\lvert{\boldsymbol{x}}\rvert^{2}\right]^{-2}(-\boldsymbol{q}^{2}+2\boldsymbol{q}\overline{\boldsymbol{x}}-2x_0\overline{\boldsymbol{x}}+\lvert{\boldsymbol{x}}\rvert^{2})f(\boldsymbol{x})dV_{I}(\boldsymbol{x})\\
&=\int _{\Omega_{I}}2\left[\boldsymbol{q}^{2}-2Re[\boldsymbol{x}]\boldsymbol{q}+\lvert{\boldsymbol{x}}\rvert^{2}\right]^{-2}(\overline{\boldsymbol{x}}-\boldsymbol{q})^{*2}f(\boldsymbol{x})dV_{I}(\boldsymbol{x}).
\end{align*}
Therefore, we has an integral representation for $\Pi_{\Omega_I}$ as the following
\begin{align*}
&\Pi_{\Omega_{I}}f(\boldsymbol{q})
=\overline{G_{\boldsymbol{q}}}T_{\Omega_{I}}f(\boldsymbol{q})\\
=&\frac{1}{\pi}\int _{\Omega_{I}}\left[\boldsymbol{q}^{2}-2Re[\boldsymbol{x}]\boldsymbol{q}+\lvert{\boldsymbol{x}}\rvert^{2}\right]^{-2}(-\boldsymbol{q}^{-2}+2\boldsymbol{q}\overline{\boldsymbol{x}}-2x_0\overline{\boldsymbol{x}}+\lvert{\boldsymbol{x}}\rvert^{2})f(\boldsymbol{x})dV_{I}(\boldsymbol{x})\\
=&-\frac{1}{\pi}\int_{\Omega_{I}} \left[\boldsymbol{q}^{2}-2Re[\boldsymbol{x}]\boldsymbol{q}+\lvert{\boldsymbol{x}}\rvert^{2}\right]^{-2}(\overline{\boldsymbol{x}}-\boldsymbol{q})^{*2}f(\boldsymbol{x})dV_{I}(\boldsymbol{x})\\
=&-\frac{1}{\pi}\int_{\Omega_{I}} (\boldsymbol{x}-\boldsymbol{q})^{-*2}f(\boldsymbol{x})dV_{I}(\boldsymbol{x}),
\end{align*}
 which completes the proof.
\end{proof}
With the previous theorem, we can easily obtain an integral representation for $\Pi_{\Omega_D}$ as follows.
\begin{corollary}\label{co1}
Let $\Omega_D \subset\mathbb{R}_*^{m+1}$ be a bounded axially symmetric domain, and $f\in C^1(\overline{\Omega_D}).$ Then, for any $ \boldsymbol{q}\in\Omega_{D} $, we have 
	$$\Pi_{\Omega_D}f(\boldsymbol{q})=\dfrac{-2}{\omega_{m-1}\pi}\int_{\Omega_{D}}\frac{(\boldsymbol{x}-\boldsymbol{q})^{-*2}}{\lvert\underline{{\boldsymbol{x}}}\rvert^{m-1}}f(\boldsymbol{x})dV(\boldsymbol{x}).$$ 

\end{corollary}
	\begin{proof}
The idea of the proof is the connection of $T_{\Omega_D}$ and $T_{\Omega_I}$ given in the proof of \cite[Lemma 3]{Ding} and the previous theorem.
Let $ \boldsymbol{x} = x_{0} + \boldsymbol{\underline{x}} \in \Omega_{D} $, we rewrite $ \boldsymbol{\underline{x}} = rI $ with $I\in \mathbb{S},\ r>0 $. 
Then, we have the volume element 
\begin{align*}
	dV(\boldsymbol{x}) = dx_{0}dV(\boldsymbol{\underline{x}}) =
	r^{m-1}dx_{0}drdS(I),
\end{align*}
where $ dS(I) $ is the surface element on the sphere $\mathbb{S}$. 
We calculate 
\begin{align*}
&\Pi_{\Omega_D}f(\boldsymbol{q})=\overline{G}T_{\Omega_D}f(\boldsymbol{q})=\overline{G}\frac{2}{\omega_{m-1}}\int_{\mathbb{S}^+}T_{\Omega_I}f(\boldsymbol{q})dS(I)\\
&=\frac{2}{\omega_{m-1}}\int_{\mathbb{S}^+}\overline{G}T_{\Omega_I}f(\boldsymbol{q})dS(I)
=\frac{-2}{\omega_{m-1}\pi}\int_{\mathbb{S}^+}\int_{\Omega_I}(\boldsymbol{x}-\boldsymbol{q})^{-*2}f(\boldsymbol{x})dV_{I}(\boldsymbol{x})dS(I)\\
&=\frac{-2}{\omega_{m-1}\pi}\int_{\Omega_D}\frac{(\boldsymbol{x}-\boldsymbol{q})^{-*2}}{\lvert\underline{{\boldsymbol{x}}}\rvert^{m-1}}f(\boldsymbol{x})dV(\boldsymbol{x}),
\end{align*}
which completes the proof.
\end{proof}
Now, we introduce some conjugate operators $\overline{T}_{\Omega_I}$, $\overline{T}_{\Omega_D}$ and $\overline{F}_{\partial\Omega_D}$ as follows. 
\begin{align*}
	&\overline{T}_{\Omega_{D}}f(\boldsymbol{q})=-\frac{1}{2\pi}\int_{\Omega_{D}}\overline{K(\boldsymbol{q},\boldsymbol{x})}f(\boldsymbol{x})dV(\boldsymbol{x}),\\
	&\overline{T}_{\Omega_I}f(\boldsymbol{q})=-\frac{1}{2\pi}\int_{\Omega_{I}}\overline{S^{-1}(\boldsymbol{q},\boldsymbol{x})}f(\boldsymbol{x})dV_I(\boldsymbol{x}),\\
	&\overline{F}_{\partial\Omega_D}f(\boldsymbol{q})=\frac{1}{2\pi }\int_{\partial\Omega_{D}}\overline{K(\boldsymbol{q},\boldsymbol{x})}n(\boldsymbol{x})f(\boldsymbol{x})d\sigma(\boldsymbol{x}).
\end{align*}
We also need the Gauss theorem for $\overline{\partial}_{\boldsymbol{x}_{I}}$ on a complex plane as follows.
\begin{theorem}
	Let $\Omega_I\subset\mathbb{C}_I$ be a domain, and $f (\boldsymbol{x}), g(\boldsymbol{x}) \in  C^1(\overline{\Omega_I}).$  Then, we have
\begin{align*}
	\int_{\Omega_{I}}(f(\boldsymbol{x})\overline{\partial}_{\boldsymbol{x}_{I}})g(\boldsymbol{x})+f(\boldsymbol{x})(\overline{\partial}_{\boldsymbol{x}_{I}}g(\boldsymbol{x}))dV_I(\boldsymbol{x})=\int_{\partial\Omega_{I}}f(\boldsymbol{x})n(\boldsymbol{x})g(\boldsymbol{x})d\sigma(\boldsymbol{x})
\end{align*}
where $n(\boldsymbol{x})$ is the outward unit normal vector on $\partial\Omega_I$ at the point $\boldsymbol{x}\in\partial\Omega_I $.		
\end{theorem}

Further, there is a Borel-Pompeiu formula given in terms of $\overline{F}_{\partial\Omega_D}$ and $\overline{T}_{\Omega_D}$ as
\begin{align}\label{BPF2}
	\overline{F}_{\partial\Omega_D}f(\boldsymbol{q})+\overline{T}_{\Omega_D}(\overline{G}f)(\boldsymbol{q})=f(\boldsymbol{q}).
\end{align}
Next, we introduce an analog of $\Pi$-operator, which turns out to be a left and right inverse of $\Pi_{\Omega_D}$ given in Corollary \ref{inverse}.
\begin{proposition}\label{conjugate}
Let $\Omega_D \subset\mathbb{R}_*^{m+1}$ be a bounded axially symmetric domain. Assume that $f\in L^{p}(\Omega_{D})$ with $p>\max\{m,2\}$. Then, the following operator 
	\begin{align*}
\Pi^+_{\Omega_D}f(\boldsymbol{q}):=G_{\boldsymbol{q}}\overline{T}_{\Omega_{D}}f(\boldsymbol{q})
	\end{align*}has an integral representation as
\begin{align*}
	\Pi^+_{\Omega_{D}}f(\boldsymbol{q})
	=&-\frac{2}{\omega_{m-1}\pi}\int_{\Omega_{D}}\frac{(\overline{\boldsymbol{x}-\boldsymbol{q}})^{-*2}f(\boldsymbol{x})}{\lvert\underline{{\boldsymbol{x}}}\rvert^{m-1}}dV(\boldsymbol{x})\\
	&+\frac{1}{\omega_{m-1}}\int_{\mathbb{S}^{+}} [\overline{\alpha}f(\boldsymbol{q}_{I})+\overline{\beta}f(\boldsymbol{q}_{-I})] dS({I})-\frac{f(\boldsymbol{q})}{2}\\
	=&\overline{\Pi_{\Omega_D}f(\bq)}+\frac{1}{\omega_{m-1}}\int_{\mathbb{S}^{+}} [\overline{\alpha}f(\boldsymbol{q}_{I})+\overline{\beta}f(\boldsymbol{q}_{-I})] dS({I})-\frac{f(\boldsymbol{q})}{2},
\end{align*}
	where 
	\begin{align*}
	\alpha=\frac{1-I_{\boldsymbol{q}}I}{2},\beta=\frac{1+I_{\boldsymbol{q}}I}{2}.
	\end{align*}
\end{proposition}

\begin{proof}
First of all, we need to obtain an expression for
\begin{align*}
\Pi^+_{\Omega_I}f(\boldsymbol{q}):=G_{\boldsymbol{q}}\overline{T}_{\Omega_{I}}f(\boldsymbol{q}).
\end{align*}
Using \cite{Ding}, we notice that $$\overline{S^{-1}(\boldsymbol{q},\boldsymbol{x})}=\overline{\alpha(\boldsymbol{x}-\boldsymbol{q}_{I})^{-1}}+\overline{\beta(\boldsymbol{x}-\boldsymbol{q}_{-I})^{-1}}=\overline{(\boldsymbol{x}-\boldsymbol{q}_{I})^{-1}}\overline{\alpha}+\overline{(\boldsymbol{x}-\boldsymbol{q}_{-I})^{-1}}\overline{\beta},$$ 
where $(\boldsymbol{x}-\boldsymbol{q}_{I})^{-1}$ is the Cauchy kernel on the plane $\mathbb{C}_{I}$,  and
\begin{align*}
	\overline{(\boldsymbol{x}-\boldsymbol{q}_{I})^{-1}}=\dfrac{\boldsymbol{x}-\boldsymbol{q}_{I}}{\lvert\boldsymbol{x}-\boldsymbol{q}_{I}\rvert^{2}}=-\overline{\partial}_{\boldsymbol{q}_{I}}\ln\lvert\boldsymbol{x}-\boldsymbol{q}_{I}\rvert=\overline{\partial}_{\boldsymbol{x}_{I}}\ln\lvert\boldsymbol{x}-\boldsymbol{q}_{I}\rvert.
	\end{align*}
Let $B_{\varepsilon}=B(\boldsymbol{q}_{I},\varepsilon)\cup B(\boldsymbol{q}_{-I},\varepsilon)\subset\Omega_{I}$ for a sufficiently small $\varepsilon>0$. 
Then we have 
\begin{align*}
&-2\pi\overline{T}_{\Omega_{I}}f(\boldsymbol{q})
=\lim_{\varepsilon\rightarrow0}\int_{\Omega_{I}\backslash B_{\varepsilon}}\overline{S^{-1}(\boldsymbol{q},\boldsymbol{x})}f(\boldsymbol{x})dV_{I}(\boldsymbol{x})\\
=&\lim_{\varepsilon\rightarrow0}\int_{\Omega_{I}\backslash B_{\varepsilon}}(\overline{(\boldsymbol{x}-\boldsymbol{q}_{I})^{-1}}\overline{\alpha}+\overline{(\boldsymbol{x}-\boldsymbol{q}_{-I})^{-1}}\overline{\beta})f(\boldsymbol{x})dV_{I}(\boldsymbol{x})\\
=&\lim_{\varepsilon\rightarrow0}\int_{\Omega_{I}\backslash B_{\varepsilon}}(\ln\lvert\boldsymbol{x}-\boldsymbol{q}_{I}\rvert\overline{\partial}_{\boldsymbol{x}_I}\overline{\alpha}+\ln\lvert\boldsymbol{x}-\boldsymbol{q}_{-I}\rvert\overline{\partial}_{\boldsymbol{x}_I}\overline{\beta})f(\boldsymbol{x})dV_{I}(\boldsymbol{x})\\
=&-\int_{\Omega_{I}}\left[ \ln\left| \boldsymbol{x}-\boldsymbol{q}_{I}\right| (\overline{\partial}_{\boldsymbol{x}_{\boldsymbol{I}}}\overline{\alpha}f(\boldsymbol{x}))+\ln\lvert\boldsymbol{x}-\boldsymbol{q}_{-I}\rvert(\overline{\partial}_{\boldsymbol{x}_{\boldsymbol{I}}}\overline{\beta}f(\boldsymbol{x}))\right] dV_{I}(\boldsymbol{x})\\
&+\int_{\partial\Omega_{I}}(\ln\lvert\boldsymbol{x}-\boldsymbol{q}_{I}\rvert
n(\boldsymbol{x})\overline{\alpha}+\ln\lvert\boldsymbol{x}-\boldsymbol{q}_{-I}\rvert n(\boldsymbol{x})\overline{\beta})f(\boldsymbol{x})d\sigma(\boldsymbol{x}).
\end{align*} 
Indeed, since $f \in  C^{1}(\overline{\Omega_{D}})$, this implies that $f$ is bounded in $\overline{\Omega_D}$. Further, the homogeneity of
\begin{align*}
	\dfrac{\partial}{\partial q_{0}}(\ln\lvert\boldsymbol{x}-\boldsymbol{q}_{I}\rvert)=\dfrac{q_{0}-x_{0}}{\lvert \boldsymbol{q}_{I}-\boldsymbol{x}\rvert^{2}},\
	\dfrac{\partial}{\partial q_{0}}(\ln\lvert\boldsymbol{x}-\boldsymbol{q}_{-I}\rvert)=\dfrac{q_{0}-x_{0}}{\lvert \boldsymbol{q}_{-I}-\boldsymbol{x}\rvert^{2}},
\end{align*}
suggest that it is integrable with respect to $\boldsymbol{x}$.
Hence we have 
\begin{align}\label{eqn1}
&2\pi\dfrac{\partial}{\partial q_{0}}\overline{T}_{\Omega_{I}}f(\boldsymbol{q})\nonumber\\
=&\int_{\Omega_{I}}\left[ \dfrac{q_{0}-x_{0}}{\lvert \boldsymbol{q}_{I}-\boldsymbol{x}\rvert^{2}}(\overline{\partial}_{\boldsymbol{x}_{\boldsymbol{I}}}\overline{\alpha}f(\boldsymbol{x}))+\dfrac{q_{0}-x_{0}}{\lvert \boldsymbol{q}_{-I}-\boldsymbol{x}\rvert^{2}}(\overline{\partial}_{\boldsymbol{x}_{\boldsymbol{I}}}\overline{\beta}f(\boldsymbol{x}))\right] dV_{I}(\boldsymbol{x}) \nonumber\\
&-\int_{\partial\Omega_{I}}\left[ \dfrac{q_{0}-x_{0}}{\lvert \boldsymbol{q}_{I}-\boldsymbol{x}\rvert^{2}}n(\boldsymbol{x})\overline{\alpha}+\dfrac{q_{0}-x_{0}}{\lvert \boldsymbol{q}_{-I}-\boldsymbol{x}\rvert^{2}}n(\boldsymbol{x})\overline{\beta}\right] f(\boldsymbol{x})d\sigma(\boldsymbol{x}).      
\end{align}  
Further, with the help of Gauss theorem, we know that
\begin{align*}
&\bigg(\int_{\partial\Omega_{I}}-\int_{\partial B_{\varepsilon}}\bigg)\left[ \dfrac{q_{0}-x_{0}}{\lvert \boldsymbol{q}_{I}-\boldsymbol{x}\rvert^{2}}n(\boldsymbol{x})\overline{\alpha}+\dfrac{q_{0}-x_{0}}{\lvert \boldsymbol{q}_{-I}-\boldsymbol{x}\rvert^{2}}n(\boldsymbol{x})\overline{\beta}\right] f(\boldsymbol{x})d\sigma(\boldsymbol{x})\\ 
=&\int_{\Omega_{I}\backslash B_{\varepsilon}}\left[ (\dfrac{q_{0}-x_{0}}{\lvert \boldsymbol{q}_{I}-\boldsymbol{x}\rvert^{2}}\overline{\partial}_{\boldsymbol{x}_{\boldsymbol{I}}})\overline{\alpha}f(\boldsymbol{x})+(\dfrac{q_{0}-x_{0}}{\lvert \boldsymbol{q}_{-I}-\boldsymbol{x}\rvert^{2}}\overline{\partial}_{\boldsymbol{x}_{\boldsymbol{I}}})\overline{\beta}f(\boldsymbol{x})\right] dV_{I}(\boldsymbol{x})\\
&+\int_{\Omega_{I}\backslash B_{\varepsilon}}\left[ \dfrac{q_{0}-x_{0}}{\lvert \boldsymbol{q}_{I}-\boldsymbol{x}\rvert^{2}}(\overline{\partial}_{\boldsymbol{x}_{\boldsymbol{I}}}\overline{\alpha}f(\boldsymbol{x}))+\dfrac{q_{0}-x_{0}}{\lvert \boldsymbol{q}_{-I}-\boldsymbol{x}\rvert^{2}}(\overline{\partial}_{\boldsymbol{x}_{I}}\overline{\beta}f(\boldsymbol{x}))\right]  dV_{I}(\boldsymbol{x}).
\end{align*}
Plugging into equation \eqref{eqn1}, we have
\begin{align*}
&2\pi\dfrac{\partial}{\partial q_{0}}\overline{T}_{\Omega_{I}}f(\boldsymbol{q})\\
=&\lim_{\varepsilon\longrightarrow0}\int_{B_{\varepsilon}}\left[ \dfrac{q_{0}-x_{0}}{\lvert \boldsymbol{q}_{I}-\boldsymbol{x}\rvert^{2}}(\overline{\partial}_{\boldsymbol{x}_{\boldsymbol{I}}}\overline{\alpha}f(\boldsymbol{x}))+\dfrac{q_{0}-x_{0}}{\lvert \boldsymbol{q}_{-I}-\boldsymbol{x}\rvert^{2}}(\overline{\partial}_{\boldsymbol{x}_{I}}\overline{\beta}f(\boldsymbol{x}))\right]  dV_{I}(\boldsymbol{x}) \\
&-\int_{\Omega_{I}\backslash B_{\varepsilon}}\left[\bigg (\dfrac{q_{0}-x_{0}}{\lvert \boldsymbol{q}_{I}-\boldsymbol{x}\rvert^{2}}\overline{\partial}_{\boldsymbol{x}_{\boldsymbol{I}}}\bigg)\overline{\alpha}f(\boldsymbol{x})+\bigg(\dfrac{q_{0}-x_{0}}{\lvert \boldsymbol{q}_{-I}-\boldsymbol{x}\rvert^{2}}\overline{\partial}_{\boldsymbol{x}_{\boldsymbol{I}}}\bigg)\overline{\beta}f(\boldsymbol{x})\right] dV_{I}(\boldsymbol{x})\\
&-\int_{\partial B_{\varepsilon}}\left[ \dfrac{q_{0}-x_{0}}{\lvert \boldsymbol{q}_{I}-\boldsymbol{x}\rvert^{2}}n(\boldsymbol{x})\overline{\alpha}+\dfrac{q_{0}-x_{0}}{\lvert \boldsymbol{q}_{-I}-\boldsymbol{x}\rvert^{2}}n(\boldsymbol{x})\overline{\beta}\right] f(\boldsymbol{x})d\sigma(\boldsymbol{x}) .
\end{align*}
From the homogeneity of $\dfrac{q_{0}-x_{0}}{\lvert \boldsymbol{q}_{I} -\boldsymbol{x}\rvert^{2}}$ and $\dfrac{q_{0}-x_{0}}{\lvert \boldsymbol{q}_{-I} -\boldsymbol{x}\rvert^{2}}$, on the one hand, one can easily show that
\begin{align*}
	\lim_{\varepsilon\rightarrow0}\int_{B_{\varepsilon}}\left[ \dfrac{q_{0}-x_{0}}{\lvert \boldsymbol{q}_{I}-\boldsymbol{x}\rvert^{2}}(\overline{\partial}_{\boldsymbol{x}_{\boldsymbol{I}}}\overline{\alpha}f(\boldsymbol{x}))+\dfrac{q_{0}-x_{0}}{\lvert \boldsymbol{q}_{-I}-\boldsymbol{x}\rvert^{2}}(\overline{\partial}_{\boldsymbol{x}_{I}}\overline{\beta}f(\boldsymbol{x}))\right]  dV_{I}(\boldsymbol{x})=0,
	\end{align*} 
and 
\begin{align*}
	\dfrac{q_{0}-x_{0}}{\lvert \boldsymbol{q}_{\pm{I}} -\boldsymbol{x}\rvert^{2}}\overline{\partial}_{\boldsymbol{x}_{I}}=\dfrac{\partial}{\partial q_{0}}\ln\lvert \boldsymbol{q}_{\pm{I}} -\boldsymbol{x}\rvert\overline{\partial}_{\boldsymbol{x}_{I}}=\dfrac{\partial}{\partial q_{0}}\overline{\partial}_{\boldsymbol{x}_{I}}\ln\lvert \boldsymbol{q}_{\pm{I}} -\boldsymbol{x}\rvert=\dfrac{\partial}{\partial q_{0}}\dfrac{1}{\overline{\boldsymbol{x}-\boldsymbol{q}_{\pm{I}} }}.
	\end{align*}
On the other hand, with a similar argument as in \cite[Theorem 8.2]{G}, we have that
\begin{align*}
	&\lim_{\varepsilon\rightarrow0}\int_{\partial B_{\varepsilon}}\left[ \dfrac{q_{0}-x_{0}}{\lvert \boldsymbol{q}_{I}-\boldsymbol{x}\rvert^{2}}n(\boldsymbol{x})\overline{\alpha}+\dfrac{q_{0}-x_{0}}{\lvert \boldsymbol{q}_{-I}-\boldsymbol{x}\rvert^{2}}n(\boldsymbol{x})\overline{\beta}\right] f(\boldsymbol{x})d\sigma(\boldsymbol{x})\\
	=&-\pi\left[\overline{\alpha} f (\boldsymbol{q}_{I} ) + \overline{\beta}f (\boldsymbol{q}_{-I} )\right].
\end{align*}
These give us that
\begin{align}\label{eqn2}
&2\pi\dfrac{\partial}{\partial q_{0}}\overline{T}_{\Omega_{I}}f(\boldsymbol{q})\nonumber\\
=&-\int_{\Omega_{I}}\dfrac{\partial}{\partial q_{0}}\left[ \dfrac{1}{\overline{\boldsymbol{x}-\boldsymbol{q}_{I} }}\overline{\alpha}+\dfrac{1}{\overline{\boldsymbol{x}-\boldsymbol{q}_{-I} }}\overline{\beta}\right] f(\boldsymbol{x})dV_{I}(\boldsymbol{x})+\pi\left[\overline{\alpha} f (\boldsymbol{q}_{I} ) + \overline{\beta} f (\boldsymbol{q}_{-I} )\right]\nonumber\\
=&-\int_{\Omega_{I}}\dfrac{\partial}{\partial q_{0}}\overline{S^{-1}(\boldsymbol{q},\boldsymbol{x})}f(\boldsymbol{x})dV_{I}(\boldsymbol{x})+\pi\left[\overline{\alpha} f (\boldsymbol{q}_I ) + \overline{\beta} f (\boldsymbol{q}_{-I} )\right].
\end{align}
Next, we consider $\dfrac{\partial}{\partial q_{i}}T_{\Omega_{I}}f(\boldsymbol{q})$, $ i = 1, . . . , m$. Firstly, we notice that
\begin{align*}
\dfrac{\partial}{\partial q_{i}}=\dfrac{\partial\zeta}{\partial{q_{i}}}\dfrac{\partial}{\partial\zeta}=\dfrac{q_{i}}{\lvert\underline{\boldsymbol{q}}\rvert}\dfrac{\partial}{\partial\zeta},\quad \dfrac{\zeta\pm\eta}{\lvert \boldsymbol{q}_{\pm{I}} -\boldsymbol{x}\rvert^{2}}=\ln \lvert \boldsymbol{q}_{\pm{I}} -\boldsymbol{x}\rvert\dfrac{\partial}{\partial{\zeta}}.
\end{align*}
Then, with a similar argument as applied to $\partial_{q_{0}}\overline{T}_{\Omega_{I}}f(\boldsymbol{q})$, we have
\begin{align} 
&2\pi\dfrac{\partial}{\partial q_{i}}\overline{T}_{\Omega_{I}}f(\boldsymbol{q})\nonumber\\
=&\dfrac{q_{i}}{\lvert\underline{\boldsymbol{q}}\rvert}\dfrac{\partial}{\partial\zeta}\int_{\Omega_{I}}\left[ \ln\left| \boldsymbol{x}-\boldsymbol{q}_{I}\right| (\overline{\partial}_{\boldsymbol{x}_{\boldsymbol{I}}}\overline{\alpha}f(\boldsymbol{x}))+\ln\lvert\boldsymbol{x}-\boldsymbol{q}_{-I}\rvert(\overline{\partial}_{\boldsymbol{x}_{\boldsymbol{I}}}\overline{\beta}f(\boldsymbol{x}))\right] dV_{I}(\boldsymbol{x})
\nonumber\\
&-\dfrac{q_{i}}{\lvert\underline{\boldsymbol{q}}\rvert}\dfrac{\partial}{\partial\zeta}\int_{\partial\Omega_{I}}(\ln\lvert\boldsymbol{x}-\boldsymbol{q}_{I}\rvert
n(\boldsymbol{x})\overline{\alpha}+\ln\lvert\boldsymbol{x}-\boldsymbol{q}_{-I}\rvert n(\boldsymbol{x})\overline{\beta})f(\boldsymbol{x})d\sigma(\boldsymbol{x})\nonumber\\
=&\dfrac{q_{i}}{\lvert\underline{\boldsymbol{q}}\rvert}\int_{\Omega_{I}}\left[ \dfrac{\zeta-\eta}{\lvert\boldsymbol{x}-\boldsymbol{q}_{I} \rvert^{2}}(\overline{\partial}_{\boldsymbol{x}_{\boldsymbol{I}}}\overline{\alpha}f(\boldsymbol{x}))+\dfrac{\zeta+\eta}{\lvert\boldsymbol{x}-\boldsymbol{q}_{-I} \rvert^{2}}(\overline{\partial}_{\boldsymbol{x}_{\boldsymbol{I}}}\overline{\beta}f(\boldsymbol{x}))\right] dV_{I}(\boldsymbol{x}) 
\nonumber\\
&-\dfrac{q_{i}}{\lvert\underline{\boldsymbol{q}}\rvert}\int_{\partial\Omega_{I}}\left[ \dfrac{\zeta-\eta}{\lvert\boldsymbol{x}-\boldsymbol{q}_{I} \rvert^{2}}n(\boldsymbol{x})\overline{\alpha}+\dfrac{\zeta+\eta}{\lvert\boldsymbol{x}-\boldsymbol{q}_{-I} \rvert^{2}}n(\boldsymbol{x})\overline{\beta}\right] f(\boldsymbol{x})d\sigma(\boldsymbol{x}).\label{eqn3} 
\end{align}
Further, Gauss theorem tells us that
\begin{align*}
&\bigg(\int_{\partial\Omega_{I}}-\int _{\partial B_{\varepsilon}}\bigg)\left[ \dfrac{\zeta-\eta}{\lvert\boldsymbol{x}-\boldsymbol{q}_{I} \rvert^{2}}n(\boldsymbol{x})\overline{\alpha}+\dfrac{\zeta+\eta}{\lvert\boldsymbol{x}-\boldsymbol{q}_{-I} \rvert^{2}}n(\boldsymbol{x})\overline{\beta}\right] f(\boldsymbol{x})d\sigma(\boldsymbol{x})\\
=&\int_{\Omega_{I}\backslash B_{\varepsilon}}\bigg[\bigg(\dfrac{\zeta-\eta}{\lvert\boldsymbol{x}-\boldsymbol{q}_{I} \rvert^{2}}\overline{\partial_{\boldsymbol{x}_{I}}}\bigg)\overline{\alpha}f(\boldsymbol{x})+\bigg(\dfrac{\zeta+\eta}{\lvert\boldsymbol{x}-\boldsymbol{q}_{-I} \rvert^{2}}\overline{\partial_{\boldsymbol{x}_{I}}}\bigg)\overline{\beta}f(\boldsymbol{x})
\\
&+\dfrac{\zeta-\eta}{\lvert\boldsymbol{x}-\boldsymbol{q}_{I} \rvert^{2}}\big(\overline{\partial_{\boldsymbol{x}_{I}}}\overline{\alpha}f(\boldsymbol{x})\big)+\dfrac{\zeta+\eta}{\lvert\boldsymbol{x}-\boldsymbol{q}_{-I} \rvert^{2}}\big(\overline{\partial_{\boldsymbol{x}_{I}}}\overline{\beta}f(\boldsymbol{x})\big)\bigg]dV_{I}(\boldsymbol{x}).
\end{align*}
Plugging into $(\ref{eqn3})$ to obtain
\begin{align*}
&2\pi\dfrac{\partial}{\partial q_{i}}\overline{T}_{\Omega_{I}}f(\boldsymbol{q})\\
=&-\dfrac{q_{i}}{\lvert\underline{\boldsymbol{q}}\rvert}\int _{\partial B_{\varepsilon}}\left[ \dfrac{\zeta-\eta}{\lvert\boldsymbol{x}-\boldsymbol{q}_{I} \rvert^{2}}n(\boldsymbol{x})\overline{\alpha}+\dfrac{\zeta+\eta}{\lvert\boldsymbol{x}-\boldsymbol{q}_{-I} \rvert^{2}}n(\boldsymbol{x})\overline{\beta}\right] f(\boldsymbol{x})d\sigma(\boldsymbol{x})\\
&+\dfrac{q_{i}}{\lvert\underline{\boldsymbol{q}}\rvert}\int_{B_{\varepsilon}}\dfrac{\zeta-\eta}{\lvert\boldsymbol{x}-\boldsymbol{q}_{I} \rvert^{2}}(\overline{\partial_{\boldsymbol{x}_{I}}}\overline{\alpha}f(\boldsymbol{x}))+\dfrac{\zeta+\eta}{\lvert\boldsymbol{x}-\boldsymbol{q}_{-I} \rvert^{2}}(\overline{\partial_{\boldsymbol{x}_{I}}}\overline{\beta}f(\boldsymbol{x}))]dV_{I}(\boldsymbol{x})\\
&-\dfrac{q_{i}}{\lvert\underline{\boldsymbol{q}}\rvert}\int_{\Omega_{I}\backslash B_{\varepsilon}}\left[ \bigg(\dfrac{\zeta-\eta}{\lvert\boldsymbol{x}-\boldsymbol{q}_{I} \rvert^{2}}\overline{\partial_{\boldsymbol{x}_{I}}}\bigg)\overline{\alpha}f(\boldsymbol{x})+\bigg(\dfrac{\zeta+\eta}{\lvert\boldsymbol{x}-\boldsymbol{q}_{-I} \rvert^{2}}\overline{\partial_{\boldsymbol{x}_{I}}}\bigg)\overline{\beta}f(\boldsymbol{x})\right] dV_{I}(\boldsymbol{x}).
\end{align*}
We also notice that
\begin{align*}
\dfrac{\zeta\pm\eta}{\lvert \boldsymbol{q}_{\pm{I}} -\boldsymbol{x}\rvert^{2}}\overline{\partial}_{\boldsymbol{x}_{\boldsymbol{I}}}
=\ln \lvert \boldsymbol{q}_{\pm{I}} -\boldsymbol{x}\rvert\dfrac{\partial}{\partial{\zeta}}\overline{\partial}_{\boldsymbol{x}_{\boldsymbol{I}}}
=\ln \lvert \boldsymbol{q}_{\pm{I}} -\boldsymbol{x}\rvert\overline{\partial}_{\boldsymbol{x}_{\boldsymbol{I}}}\dfrac{\partial}{\partial{\zeta}}
=\dfrac{1}{\overline{\boldsymbol{x}-\boldsymbol{q}_{\pm{I}} }}\dfrac{\partial}{\partial{\zeta}},
\end{align*}
which gives us that
\begin{align*}
  &2\pi\dfrac{\partial}{\partial q_{i}}\overline{T}_{\Omega_{I}}f(\boldsymbol{q})\\
  =&\lim_{\varepsilon\rightarrow0}\dfrac{q_{i}}{\lvert\underline{\boldsymbol{q}}\rvert}\bigg[-\int_{\Omega_{I}}\left[\bigg (\dfrac{\partial}{\partial{\zeta}}\dfrac{1}{\overline{\boldsymbol{x}-\boldsymbol{q}_{I} }}\bigg)\overline{\alpha}f(\boldsymbol{x})+\bigg(\dfrac{\partial}{\partial{\zeta}}\dfrac{1}{\overline{\boldsymbol{x}-\boldsymbol{q}_{-I} }}\bigg)\overline{\beta}f(\boldsymbol{x})\right] dV_{I}(\boldsymbol{x})\\
  &-\dfrac{q_{i}}{\lvert\underline{\boldsymbol{q}}\rvert}\int _{\partial B_{\varepsilon}}\left[ \dfrac{\zeta-\eta}{\lvert\boldsymbol{x}-\boldsymbol{q}_{I} \rvert^{2}}n(\boldsymbol{x})\overline{\alpha}f(\boldsymbol{x})+\dfrac{\zeta+\eta}{\lvert\boldsymbol{x}-\boldsymbol{q}_{-I} \rvert^{2}}n(\boldsymbol{x})\overline{\beta}f(\boldsymbol{x})\right] d\sigma(\boldsymbol{x})\\
  =&\dfrac{q_{i}}{\lvert\underline{\boldsymbol{q}}\rvert}\bigg[-\int_{\Omega_{I}}\dfrac{\partial}{\partial{\zeta}}\overline{{S}^{-1}(\boldsymbol{q},\boldsymbol{x})}f(\boldsymbol{x})dV_{I}(\boldsymbol{x})\\
 &-\lim_{\varepsilon\rightarrow0}\int _{\partial B_{\varepsilon}} \dfrac{\zeta-\eta}{\lvert\boldsymbol{x}-\boldsymbol{q}_{I} \rvert^{2}}n(\boldsymbol{x})\overline{\alpha}f(\boldsymbol{x})+\dfrac{\zeta+\eta}{\lvert\boldsymbol{x}-\boldsymbol{q}_{-I} \rvert^{2}}n(\boldsymbol{x})\overline{\beta}f(\boldsymbol{x})d\sigma(\boldsymbol{x})\bigg] \\
=&\dfrac{q_{i}}{\lvert\underline{\boldsymbol{q}}\rvert}\left[-\int_{\Omega_{I}}\dfrac{\partial}{\partial{\zeta}}\overline{{S}^{-1}(\boldsymbol{q},\boldsymbol{x})}f(\boldsymbol{x})dV_{I}(\boldsymbol{x})-\pi(-I\overline{\alpha}f(\boldsymbol{q}_{I})+I\overline{\beta}f(\boldsymbol{q}_{-I}))\right].
\end{align*}
Therefore, we obtain
\begin{align}
&2\pi\dfrac{\underline{\boldsymbol{q}}}{\lvert\underline{\boldsymbol{q}}\rvert^{2}}\sum_{i=1}^{m}q_{i}\dfrac{\partial}{\partial q_{i}}\overline{T}_{\Omega_{I}}f(\boldsymbol{q})\nonumber\\
=&\dfrac{\underline{\boldsymbol{q}}}{\lvert\underline{\boldsymbol{q}}\rvert^{2}}\sum_{i=1}^{m}q_{i}\dfrac{q_{i}}{\lvert\underline{\boldsymbol{q}}\rvert}\left[-\int_{\Omega_{I}}\dfrac{\partial}{\partial{\zeta}}\overline{S^{-1}(\boldsymbol{q},\boldsymbol{x})}f(\boldsymbol{x})dV_{I}(\boldsymbol{x})-\pi(-I\overline{\alpha}f(\boldsymbol{q}_{I})+I\overline{\beta}f(\boldsymbol{q}_{-I}))\right]\nonumber\\
=&-\int_{\Omega_{I}}I_{\boldsymbol{q}}\dfrac{\partial}{\partial \zeta}\overline{S^{-1}(\boldsymbol{q},\boldsymbol{x})}f(\boldsymbol{x})dV_{I}(\boldsymbol{x})-\pi I_{\boldsymbol{q}}(-I\overline{\alpha}f(\boldsymbol{q}_{I})+I\overline{\beta}f(\boldsymbol{q}_{-I}))\nonumber\\
=&-\int_{\Omega_{I}}I_{\boldsymbol{q}}\dfrac{\partial}{\partial \zeta}\overline{S^{-1}(\boldsymbol{q},\boldsymbol{x})}f(\boldsymbol{x})dV_{I}(\boldsymbol{x})-\pi (\alpha f(\boldsymbol{q}_{I})+\beta f(\boldsymbol{q}_{-I}))\nonumber\\
=&-\int_{\Omega_{I}}I_{\boldsymbol{q}}\dfrac{\partial}{\partial \zeta}\overline{S^{-1}(\boldsymbol{q},\boldsymbol{x})}f(\boldsymbol{x})dV_{I}(\boldsymbol{x})-\pi f(\boldsymbol{q}).\label{eqn4}
\end{align}
Combining \eqref{eqn2} and \eqref{eqn4}, we have that
\begin{align*}
&2\pi\Pi^+_{\Omega_{I}}f(\boldsymbol{q})=2\pi G_{\boldsymbol{q}}\overline{T}_{\Omega_{I}}f(\boldsymbol{q})=2\pi\dfrac{\partial}{\partial q_{0}}\overline{T}_{\Omega_{I}}f(\boldsymbol{q})+2\pi\dfrac{\underline{\boldsymbol{q}}}{\lvert\underline{\boldsymbol{q}}\rvert^{2}}\sum_{i=1}^{m}q_{i}\dfrac{\partial}{\partial q_{i}}\overline{T}_{\Omega_{I}}f(\boldsymbol{q})\\
=&-\int_{\Omega_{I}}\bigg(\dfrac{\partial}{\partial q_{0}}+I_{\boldsymbol{q}}\dfrac{\partial}{\partial \zeta}\bigg)\overline{S^{-1}(\boldsymbol{q},\boldsymbol{x})}f(\boldsymbol{x})dV_{I}(\boldsymbol{x})\\
&+\pi\big(\overline{\alpha}f(\boldsymbol{q}_{I})+\overline{\beta} f(\boldsymbol{q}_{-I})\big)- \pi f(\boldsymbol{q}).
\end{align*}
Then, similarly to the proof of Theorem \ref{Thm3.2}, we have that
\begin{align*}
	&\dfrac{\partial}{\partial q_{0}}\overline{S^{-1}(\boldsymbol{q},\boldsymbol{x})}=-\dfrac{\partial}{\partial q_{0}}(\overline{\boldsymbol{q}}-\boldsymbol{x})(\overline{\boldsymbol{q}}^{2}-2Re[\boldsymbol{x}]\overline{\boldsymbol{q}}+\lvert{\boldsymbol{x}}\rvert^{2})^{-1}	\\
	&=-\dfrac{\partial}{\partial q_{0}}\left[(q_0-I_{\boldsymbol{q}}\zeta)-\boldsymbol{x}\right]\left[(q_0-I_{\boldsymbol{q}}\zeta)^2-2x_0(q_0-I_{\boldsymbol{q}}\zeta)+\lvert \boldsymbol{x}\rvert^{2}\right]^{-1}\\
	&=-\left[a_0 ^{-1}-2b_0 a_0^{-2}(\overline{\boldsymbol{q}}-x_0)\right],
\end{align*}
and 
\begin{align*}
	&I_{\boldsymbol{q}}\dfrac{\partial}{\partial \zeta}\overline{S^{-1}(\boldsymbol{q},\boldsymbol{x})}=-I_{\boldsymbol{q}}\dfrac{\partial}{\partial \zeta}(\overline{\boldsymbol{q}}-\boldsymbol{x})(\overline{\boldsymbol{q}}^{2}-2Re[\boldsymbol{x}]\overline{\boldsymbol{q}}+\lvert{\boldsymbol{x}}\rvert^{2})^{-1}\\
	&=-I_{\boldsymbol{q}}\dfrac{\partial}{\partial \zeta}\left[(q_0-I_{\boldsymbol{q}}\zeta)-\boldsymbol{x}\right]\left[(q_0-I_{\boldsymbol{q}}\zeta)^2-2x_0(q_0-I_{\boldsymbol{q}}\zeta)+\lvert \boldsymbol{x}\rvert^{2}\right]^{-1}\\
	&=-\left[a_0 ^{-1}-2b_0 a_0^{-2}(\overline{\boldsymbol{q}}-x_0)\right],
\end{align*}
where 
\begin{align*}a_0&=\overline{\boldsymbol{q}}^{2}-2Re[\boldsymbol{x}]\overline{\boldsymbol{q}}+\lvert{\boldsymbol{x}}\rvert^{2}=(q_0-I_{\boldsymbol{q}}\zeta)^2-2x_0(q_0-I_{\boldsymbol{q}}\zeta)+\lvert \boldsymbol{x}\rvert^{2},\\
	b_0&=\overline{\boldsymbol{q}}-\boldsymbol{x}=(q_0-I_{\boldsymbol{q}}\zeta)-\boldsymbol{x}.
\end{align*}
Hence, we have
\begin{align*}
 &2\pi\Pi^+_{\Omega_{I}}f(\boldsymbol{q})=2\pi G_{\boldsymbol{q}}\overline{T}_{\Omega_{I}}f(\boldsymbol{q})\\
 =&-\int_{\Omega_{I}}\bigg(\dfrac{\partial}{\partial q_{0}}+I_{\boldsymbol{q}}\dfrac{\partial}{\partial \zeta}\bigg)\overline{S^{-1}(\boldsymbol{q},\boldsymbol{x})}f(\boldsymbol{x})dV_{I}(\boldsymbol{x})\\
 &+\pi\big(\overline{\alpha}f(\boldsymbol{q}_{I})+\overline{\beta} f(\boldsymbol{q}_{-I})\big)- \pi f(\boldsymbol{q})\\
 =&2\int_{\Omega_{I}}\left[a_0 ^{-1}-2b_0 a_0^{-2}(\overline{\boldsymbol{q}}-x_0)\right]f(\boldsymbol{x})dV_{I}(\boldsymbol{x})\\
  &+\pi\big(\overline{\alpha}f(\boldsymbol{q}_{I})+\overline{\beta} f(\boldsymbol{q}_{-I})\big)- \pi f(\boldsymbol{q})\\
  =&2\int_{\Omega_{I}}\left[ a_0^{-1} -2(\overline{\boldsymbol{q}}-\boldsymbol{x})(\overline{\boldsymbol{q}}^{2}-2Re[\boldsymbol{x}]\overline{\boldsymbol{q}}+\lvert{\boldsymbol{x}}\rvert^{2})^{-2}(\overline{\boldsymbol{q}}-x_0)\right]f(\boldsymbol{x})dV_{I}(\boldsymbol{x})\\
   &+\pi\big(\overline{\alpha}f(\boldsymbol{q}_{I})+\overline{\beta} f(\boldsymbol{q}_{-I})\big)- \pi f(\boldsymbol{q})\\
  =& 2\int_{\Omega_{I}}\left[ a_0^{-1} -2(\overline{\boldsymbol{q}}-\boldsymbol{x})(\overline{\boldsymbol{q}}-x_0)(\overline{\boldsymbol{q}}^{2}-2Re[\boldsymbol{x}]\overline{\boldsymbol{q}}+\lvert{\boldsymbol{x}}\rvert^{2})^{-2}\right]f(\boldsymbol{x})dV_{I}(\boldsymbol{x})\\
   &+\pi\big(\overline{\alpha}f(\boldsymbol{q}_{I})+\overline{\beta} f(\boldsymbol{q}_{-I})\big)- \pi f(\boldsymbol{q})\\
  =& 2\int_{\Omega_{I}}\left[ -\overline{\boldsymbol{q}}^{2}+2\boldsymbol{x}\overline{\boldsymbol{q}}-2x_0\boldsymbol{x}+|\boldsymbol{x}|^2 \right] (\overline{\boldsymbol{q}}^{2}-2Re[\boldsymbol{x}]\overline{\boldsymbol{q}}+\lvert{\boldsymbol{x}}\rvert^{2})^{-2}f(\boldsymbol{x})dV_{I}(\boldsymbol{x})\\
   &+\pi\big(\overline{\alpha}f(\boldsymbol{q}_{I})+\overline{\beta} f(\boldsymbol{q}_{-I})\big)- \pi f(\boldsymbol{q})\\
  =& 2\int_{\Omega_{I}}-(\boldsymbol{x}-\overline{\boldsymbol{q}})^{*2}(\overline{\boldsymbol{q}}^{2}-2Re[\boldsymbol{x}]\overline{\boldsymbol{q}}+\lvert{\boldsymbol{x}}\rvert^{2})^{-2}f(\boldsymbol{x})dV_{I}(\boldsymbol{x})\\ 
  &+\pi\big(\overline{\alpha}f(\boldsymbol{q}_{I})+\overline{\beta} f(\boldsymbol{q}_{-I})\big)- \pi f(\boldsymbol{q})\\
   =& 2\int_{\Omega_{I}}-(\overline{\boldsymbol{x}-\boldsymbol{q}})^{-*2}f(\boldsymbol{x})dV_{I}(\boldsymbol{x})+\pi\big(\overline{\alpha}f(\boldsymbol{q}_{I})+\overline{\beta} f(\boldsymbol{q}_{-I})\big)- \pi f(\boldsymbol{q}).
\end{align*}
Now, we can apply a similar argument as in the proof of Corollary \ref{co1} to get
 \begin{align*}
&\overline{T}_{\Omega_{D}}f(\boldsymbol{q})=-\dfrac{1}{2\pi}\int_{\Omega_{D}}\overline{K(\boldsymbol{q},\boldsymbol{x})}f(\boldsymbol{x})dV(\boldsymbol{x})\\
=&-\dfrac{1}{2\pi}\int_{\mathbb{S}^{+}}\int _{\Omega_{I}}\overline{K(\boldsymbol{q},\boldsymbol{x})}f(\boldsymbol{x})r^{m-1}dx_{0}drdS(I)\\
=&\dfrac{-1}{\pi\omega_{m-1}}\int_{\mathbb{S}^{+}}\int _{\Omega_{I}}\overline{S^{-1}(\boldsymbol{q},\boldsymbol{x})}dV_{{I}}(\boldsymbol{x})dS({I})
=\frac{2}{\omega_{m-1}}\int_{\mathbb{S}^{+}}\overline{T}_{\Omega_{I}}f(\boldsymbol{q})d\textit{S}(I).
 \end{align*}
Using \cite[Lemma 3.5, Theorem 3.4]{Ding}, we get 
 \begin{align*}
\dfrac{\partial}{\partial q_{i}}\overline{T}_{\Omega_{D}}f(\boldsymbol{q})=\int_{\mathbb{S}^{+}}\dfrac{\partial}{\partial q_{i}}\overline{T}_{\Omega_{I}}f(\boldsymbol{q})dS({I}),
 \end{align*}
and
 \begin{align*}
	2\pi\Pi^+_{\Omega_D}f(\boldsymbol{q})=&2\pi G_{\boldsymbol{q}}\overline{T}_{\Omega_{D}}
=2\pi G_{\boldsymbol{q}}\frac{2}{\omega_{m-1}}\int_{\mathbb{S}^{+}}\overline{T}_{\Omega_{I}}f(\boldsymbol{q})dS({I})\\
=&\frac{2}{\omega_{m-1}}\int_{\mathbb{S}^{+}}2\pi G_{\boldsymbol{q}}\overline{T}_{\Omega_{I}}f(\boldsymbol{q})dS({I}).
 \end{align*}
Thus, we have 
\begin{align*}
	2\pi\Pi^+_{\Omega_{D}}f(\boldsymbol{q})
	=&-\frac{2}{\omega_{m-1}}\int_{\mathbb{S}^{+}}2\int_{\Omega_{I}}(\overline{\boldsymbol{x}-\boldsymbol{q}})^{-*2}f(\boldsymbol{x})dV_{I}(\boldsymbol{x})dS({I})\\
	&+\frac{2}{\omega_{m-1}}\int_{\mathbb{S}^{+}}\bigg( \pi[\overline{\alpha}f(\boldsymbol{q}_{I})+\overline{\beta}f(\boldsymbol{q}_{-I})]-\pi f(\boldsymbol{q})\bigg) dS({I})\\
	=&-\frac{4}{\omega_{m-1}}\int_{\Omega_D}\frac{(\overline{\boldsymbol{x}-\boldsymbol{q}})^{-*2}f(\boldsymbol{x})}{\lvert\underline{{\boldsymbol{x}}}\rvert^{m-1}}dV(\boldsymbol{x})\\
	&+\frac{2}{\omega_{m-1}}\int_{\mathbb{S}^{+}} \pi[\overline{\alpha}f(\boldsymbol{q}_{I})+\overline{\beta}f(\boldsymbol{q}_{-I})] dS({I})-\pi f(\boldsymbol{q}),
 \end{align*}
which gives us an integral expression for $\Pi^+_{\Omega_D}$ as 
\begin{align*}
\Pi^+_{\Omega_{D}}f(\boldsymbol{q})=&-\frac{2}{\omega_{m-1}\pi}\int_{\Omega_{D}}\frac{(\overline{\boldsymbol{x}-\boldsymbol{q}})^{-*2}f(\boldsymbol{x})}{\lvert\underline{{\boldsymbol{x}}}\rvert^{m-1}}dV(\boldsymbol{x})\\
&+\frac{1}{\omega_{m-1}}\int_{\mathbb{S}^{+}} [\overline{\alpha}f(\boldsymbol{q}_{I})+\overline{\beta}f(\boldsymbol{q}_{-I})] dS({I})-\frac{f(\boldsymbol{q})}{2}.
\end{align*}
With the calculation given in Theorem \ref{Thm3.2} and Corollary \ref{co1}, we can see that 
\begin{align*}
-\frac{2}{\omega_{m-1}\pi}\int_{\Omega_{D}}\frac{(\overline{\boldsymbol{x}-\boldsymbol{q}})^{-*2}f(\boldsymbol{x})}{\lvert\underline{{\boldsymbol{x}}}\rvert^{m-1}}dV(\boldsymbol{x})
\end{align*}
is the conjugation of $\Pi_{\Omega_{D}}f$, which completes the proof.
\end{proof}
 \begin{remark}
 Since $\overline{\alpha}\neq\alpha$ and $\overline{\beta}\neq\beta$, then $\overline{\alpha}f(\boldsymbol{q}_{I})+\overline{\beta}f(\boldsymbol{q}_{-I})\neq f(\boldsymbol{q})$. Hence, the last term for $\Pi_{\Omega_D}^+$ can not be canceled out even for $f\in\mathcal{S}({\Omega_D})$. However, it can be canceled out when $n=2$, since $I_q=I$ and $\alpha=1,$ $\beta=0$ when $n=2$, the result reduces to the monogenic case given in \cite{G2}. The condition $p>\{m,2\}$ comes from the existence of $T_{\Omega_D}f$ for $f\in L^p(\Omega_D)$, see \cite[Proposition 3.1]{Ding}.
 \end{remark}
 Now, we introduce a weighted measure $d\mu(\bx)=|\ubx|^{1-m}dV(\bx)$. Then, the weighted $L^p(\Omega_D,d\mu)$ stands for the space of measurable functions with
 \begin{align*} 
\bigg\{f\in L^p(\Omega_D,d\mu):\  \|f\|_{L^p(d\mu)}^p:=\int_{\Omega_D}|f(\bx)|^pd\mu(\bx)<\infty\bigg\}.
 \end{align*}
 It is easy to observe that $L^p(\Omega_D,d\mu)\subset L^p(\Omega_D)$, when $\Omega_D\subset\mathbb{R}_*^{m+1}$ is bounded. We also denote $$\mathcal{L}^p(\Omega_D,d\mu):=L^p(\Omega_D,d\mu)\cap \mathcal{S}(\Omega_D).$$ Similarly, we have weighted Sobolev spaces denoted by $W^k_p(\Omega_D,d\mu)$, and $\mathring{W}^k_p(\Omega_D,d\mu)$ stands for the space of functions in $W^k_p(\Omega_D,d\mu)$ with compact support in $\Omega_D$. Now we introduce a mapping property of $\Pi_{\Omega_{D}}$ on $L^p(\Omega_D,d\mu)$ as follows.
\begin{proposition}\label{pr3.6}
	Let $\Omega_D \subset\mathbb{R}_*^{m+1}$ be a bounded axially symmetric domain and $1<p<\infty$. Then we have that 	
	$$\Pi_{\Omega_{D}}:L^p (\Omega_D,d\mu)\longrightarrow L^p (\Omega_D,d\mu).$$
\end{proposition}
\begin{proof}
	From the definition of $\Pi_{\Omega_{D}}$, we get
\begin{align*}
	\left|\Pi_{\Omega_{D}}f(\boldsymbol{q}) \right|&=\left| 	\frac{-2}{\omega_{m-1}\pi}\int_{\Omega_D}\frac{(\boldsymbol{x}-\boldsymbol{q})^{-*2}}{\lvert\underline{{\boldsymbol{x}}}\rvert^{m-1}}f(\boldsymbol{x})dV(\boldsymbol{x})\right| \\
	&=C_1\left| \int_{\Omega_D}\bigg(\alpha\frac{1}{(\boldsymbol{x}-\boldsymbol{q}_I)^2}+\beta\frac{1}{(\boldsymbol{x}-\boldsymbol{q}_{-I})^2}\bigg)f( \boldsymbol{x})\left|\underline{\boldsymbol{x}} \right|^{-m+1} dV(\boldsymbol{x})\right| 
\end{align*}
where $C_1=\frac{2}{\omega_{m-1}\pi}$. Let 
\begin{align*}
 A_{\Omega_I}f( \boldsymbol{q})&=\int_{\Omega_{I}}(\boldsymbol{x}-\boldsymbol{q})^{-*2}f( \boldsymbol{x})dV_I(\boldsymbol{x})\\
 &=\int_{\Omega_{I}}\bigg(\alpha\frac{1}{(\boldsymbol{x}-\boldsymbol{q}_I)^2}+\beta\frac{1}{(\boldsymbol{x}-\boldsymbol{q}_{-I})^2}\bigg)f( \boldsymbol{x})dV_I(\boldsymbol{x}).
\end{align*}
 Hence, applying the Calderon-Zygmund Theorem in \cite[Chapter XI, 3.1]{Mi} to the complex plane case, we have 
 \begin{align*}
 	||A_{\Omega_I}f||_{L^p(\Omega_I)}\leq c|| f||_{L^p(\Omega_I)},\ 1<p<\infty,
 \end{align*}
 in other words,
 \begin{align*}
 	\int_{\Omega_{I}}\left| \int_{\Omega_{I}}\frac{1}{(\boldsymbol{x}-\boldsymbol{q}_{\pm I})^2}f( \boldsymbol{x})dV_I(\boldsymbol{x})\right|^p 	dV_I(\boldsymbol{q})\leq c\int_{\Omega_{I}}\left|f \right| ^p dV_I(\boldsymbol{q}),
 \end{align*}
 where $c$ is a constant which does not depend on $I$. In our case, the symbol of the singular integral operator
  \begin{align*}
  	A_{\Omega_I}f( \boldsymbol{q})=\int_{\Omega_{I}}\frac{1}{(\boldsymbol{x}-\boldsymbol{q}_I)^{2}}f( \boldsymbol{x})dV_I(\boldsymbol{x})
\end{align*}
is defined by 
$$\Phi(\theta)=\mathcal{F}\bigg(\dfrac{1}{(\boldsymbol{x}-\boldsymbol{q}_I)^{2}}\bigg)(\theta)$$
where  $\mathcal{F}\bigg(\dfrac{1}{(\boldsymbol{x}-\boldsymbol{q}_I)^{2}}\bigg)(\theta)$ is the Fourier transform of $\dfrac{1}{\zeta^2}=\dfrac{1}{(\boldsymbol{x}-\boldsymbol{q}_I)^{2}}$ and $\theta=\dfrac{\boldsymbol{x}}{\left|\boldsymbol{x} \right| }$. Using the representation of the symbol by the characteristic $\kappa(\theta^{'})=\dfrac{1}{\zeta^2}\left|\zeta \right|^2, \theta^{'}=\dfrac{\zeta}{ \left|\zeta \right|} ,$ we get
 \begin{align*}
 \Phi(\theta)=\int_{S}\kappa(\theta^{'})\left[\ln\frac{1}{\left| \cos\gamma\right| }+\dfrac{i\pi}{2} \text{sign} \cos \gamma\right] dS_{\theta^{'}},
\end{align*}
where $S$ is the unit circle in $\mathbb{R}^2$. We denote by $\gamma$ the angle between $\theta$ and $\theta^{'}$. For the characteristic $\kappa(\theta^{'})$ we have $\kappa(\theta^{'})=\overline{\theta^{'}}^2.$
We also know  \begin{align*}
	A_{\Omega_I}:L^p (\Omega_I)\longrightarrow L^p (\Omega_I)	,
\end{align*}
so by the help of the theory of Calderon-Zygmund in \cite[Chapter XI,9.1]{Mi},  we obtain
\begin{align*}
	|| A_{\Omega_I}||_{L_p (\Omega_I)}=\text{supess}\left|\Phi(\theta) \right| .
\end{align*}
Hence, we get the estimate
\begin{align*}
	\left| \Phi(\theta)\right|^2&= 	\left|\int_{S}\kappa(\theta^{'})\left[\ln\frac{1}{\left| \cos\gamma\right| }+\dfrac{i\pi}{2} \text{sign} \cos \gamma\right] dS_{\theta^{'}} \right|^2\\
	&\leq \int_{S}\left| \kappa(\theta^{'}) \right|^2\left|\ln\frac{1}{\left| \cos\gamma\right| }+\dfrac{i\pi}{2} \text{sign} \cos \gamma \right|^2dS_{\theta^{'}}\\
	&\leq (\int_{S}\left| \kappa(\theta^{'}) \right|^4dS_{\theta^{'}})^{\frac{1}{2}}(\int_{S}\left|\ln\frac{1}{\left| \cos\gamma\right| }+\dfrac{i\pi}{2} \text{sign} \cos \gamma\right|^4 dS_{\theta^{'}})^{\frac{1}{2}}\\
	&\leq (\int_{S}\left|\kappa(\theta^{'})   \right|^4dS_{\theta^{'}})^{\frac{1}{2}} C'=(\int_{S}\left| \overline{\theta^{'}} \right|^4dS_{\theta^{'}})^{\frac{1}{2}}C'\\
	&\leq(\int_{S}1dS_{\theta^{'}})^{\frac{1}{2}}C'=(2\pi)^{\frac{1}{2}}C',
\end{align*}where 
\begin{align*}	C'=\left( \int_{S}\left|\ln\frac{1}{\left| \cos\gamma\right| }+\dfrac{i\pi}{2} \text{sign} \cos \gamma\right|^4 dS_{\theta^{'}}\right) ^{\frac{1}{2}}.
\end{align*}
Therefore, we obtain an estimate of the norm of the $A_{\Omega_I}$	in the form 
\begin{align*}	
	|| A_{\Omega_I}f||_{L^p(\Omega_I)}\leq c|| f||_{L^p(\Omega_I)},	
\end{align*}
where $c=(2\pi)^{\frac{1}{4}}C'^{\frac{1}{2}}$. Hence, we have 
	\begin{align*} 
		&\parallel \Pi_{\Omega_{D}}f(\boldsymbol{q})\parallel^p_{L^p(d\mu)}=\int_{\Omega_{D}}\left|\Pi_{\Omega_{D}}f(\boldsymbol{q}) \right|^pd\mu(\boldsymbol{q})\\
		\leq&C_1\int_{\Omega_{D}}\left|\int_{\Omega_{D}}(\alpha\frac{1}{(\boldsymbol{x}-\boldsymbol{q}_I)^2}+\beta\frac{1}{(\boldsymbol{x}-\boldsymbol{q}_{-I})^2})f( \boldsymbol{x})\left|\underline{\boldsymbol{x}} \right|^{-m+1}dV(\boldsymbol{x}) \right|^pd\mu(\boldsymbol{q})\\
		\leq&C_1\int_{\Omega_{D}}	\int_{\mathbb{S}^+}\left|\int_{\Omega_I}(\alpha\frac{1}{(\boldsymbol{x}-\boldsymbol{q}_I)^2}+\beta\frac{1}{(\boldsymbol{x}-\boldsymbol{q}_{-I})^2})f( \boldsymbol{x})	dV_I (\boldsymbol{x}) \right|^pdS(I)d\mu(\boldsymbol{q})\\	
		\leq&C_1\int_{\mathbb{S}^+}\int_{\Omega_{J}}\int_{\mathbb{S}^+}\left|\int_{\Omega_{I}}(\alpha\frac{1}{(\boldsymbol{x}-\boldsymbol{q}_I)^2}+\beta\frac{1}{(\boldsymbol{x}-\boldsymbol{q}_{-I})^2})f( \boldsymbol{x})	dV_I (\boldsymbol{x}) \right|^pdS(I)\\
	&\cdot dV_J(\boldsymbol{q})dS(J)\\
		\leq&C_1\int_{\mathbb{S}^+}\int_{\mathbb{S}^+}\int_{\Omega_{J}}\left| \int_{\Omega_{I}}\frac{1}{(\boldsymbol{x}-\boldsymbol{q}_I)^2}f( \boldsymbol{x})dV_I (\boldsymbol{x})\right|^pdV_J(\boldsymbol{q})dS(I)dS(J)\\
		&+C_1\int_{\mathbb{S}^+}\int_{\mathbb{S}^+}\int_{\Omega_{J}}\left| \int_{\Omega_{I}}\frac{1}{(\boldsymbol{x}-\boldsymbol{q}_{-I})^2}f( \boldsymbol{x})dV_I (\boldsymbol{x})\right|^p dV_J(\boldsymbol{q})dS(I)dS(J).   
\end{align*}
We notice that the integrand $\dfrac{1}{(\boldsymbol{x}-\boldsymbol{q}_{\pm I})^2}$ only depend on $q_0,\left|\underline{\boldsymbol{q}} \right|$, and $\Omega_I=\Omega_J$ up to a rotation around the real axis. This allows us to rewrite the last two integrals above as  
\begin{align*}
	&C_1\int_{\mathbb{S}^+}\int_{\mathbb{S}^+}\int_{\Omega_{J}}\left| \int_{\Omega_{I}}\frac{1}{(\boldsymbol{x}-\boldsymbol{q}_{\pm I})^2}f( \boldsymbol{x})dV_I (\boldsymbol{x})\right|^p dV_J(\boldsymbol{q})dS(I)dS(J)\\
&=\int_{\mathbb{S}^+}\int_{\Omega_{I}}\left| \int_{\Omega_{I}}\frac{1}{(\boldsymbol{x}-\boldsymbol{q}_{\pm I})^2}f( \boldsymbol{x})dV_I (\boldsymbol{x})\right|^p dV_I(\boldsymbol{q})dS(I)\\
&\leq cC_1 \int_{\mathbb{S}^+}\int_{\Omega_{I}}\left|f( \boldsymbol{q}) \right|^pdV_I(\boldsymbol{q})dS(I).
\end{align*} 
Therefore, we get 
\begin{align*}
		&|| \Pi_{\Omega_{D}}f(\boldsymbol{q})||^p_{L^p(d\mu)}
		\leq 2cC_1 \int_{\mathbb{S}^+}\int_{\Omega_{I}}\left|f( \boldsymbol{q}) \right|^pdV_I(\boldsymbol{q})dS(I)\nonumber\\
		&=2cC_1 \int_{\Omega_{D}}\left|f( \boldsymbol{q}) \right|^pd\mu(\boldsymbol{q})
		=2cC_1  || f(\boldsymbol{q})||^p_{L^p(d\mu)},
\end{align*}
where
\begin{align*}
2cC_1 =2(2\pi)^{\frac{1}{4}}C'^{\frac{1}{2}}\frac{2}{\omega_{m-1}\pi}
=\frac{4(2\pi)^{\frac{1}{4}}C'^{\frac{1}{2}}}{\omega_{m-1}\pi},
\end{align*}
which completes the proof.
	\end{proof}
\begin{remark}
In the proof above, we actually have a norm estimate for the $L^p(d\mu)$ norm of $\Pi_{\Omega_D}$ as the following
	\begin{align} \label{eq3.5}
	|| \Pi_{\Omega_{D}}f ||_{L^p(d\mu)}	\leq C	|| f ||_{L^p(d\mu)},
	\end{align}
	where
	\begin{align*}
	C=\bigg(\frac{4(2\pi)^{\frac{1}{4}}C'^{\frac{1}{2}}}{\omega_{m-1}\pi}\bigg)^{\frac{1}{p}},\ 
		C'=\left( \int_{S}\left|\ln\frac{1}{\left| \cos\gamma\right| }+\dfrac{i\pi}{2} \text{sign} \cos \gamma\right|^4 dS_{\theta^{'}}\right) ^{\frac{1}{2}}.
	\end{align*}
	One might also notice that $$T_{\Omega_D}:\mathcal{L}^p(\Omega_D,d\mu)\longrightarrow\mathcal{L}^p(\Omega_D,d\mu),$$ where $\mathcal{L}^p(\Omega_D,d\mu)={L}^p(\Omega_D,d\mu)\cap\mathcal{S}(\Omega_D)$, see \cite{Ding}. However, $\Pi_{\Omega_D}$ does not map $\mathcal{L}^p(\Omega_D,d\mu)$ to itself, since $G$ does not map $\mathcal{S}(\Omega_D)$ to itself, which can be easily checked by definition.
\end{remark}
Next, we introduce some properties for  $\Pi_{\Omega_{D}}$ as follows. 
	\begin{theorem}\label{th3.8}
		Let $\Omega_D \subset\mathbb{R}_*^{m+1}$ be a bounded axially symmetric domain. Suppose that $f\in {W}_{p}^{1}(\Omega_{D},d\mu)$ with $1<p<\infty$. Then, we have
	\begin{enumerate}
\item 	$G\Pi_{\Omega_{D}}f=\overline{G}f,$ \label{1}
\item $\Pi_{\Omega_{D}}Gf=\overline{G}f-\overline{G}F_{\partial\Omega_{D}}f,$ \label{2}
	\item 	$(G\Pi_{\Omega_{D}}-\Pi_{\Omega_{D}}G)f=\overline{G}F_{\partial\Omega_{D}}f$.
		 \end{enumerate}	 
	\end{theorem}
\begin{proof}
	From the definition of $\Pi_{\Omega_{D}}$ and $GT_{\Omega_{D}}=I$, we have
\begin{align*}
		G\Pi_{\Omega_{D}}f=G\overline{G}T_{\Omega_{D}}f=\overline{G}GT_{\Omega_{D}}f=\overline{G}f.
	\end{align*}

Using the Borel-Pompeiu formula \eqref{BPF1}, we obtain	
	\begin{align*}	
		\Pi_{\Omega_{D}}Gf=\overline{G}T_{\Omega_{D}}Gf=\overline{G}(I-F_{\partial\Omega_{D}})f=\overline{G}f-\overline{G}F_{\partial\Omega_{D}}f.
		\end{align*}
		With the statements of \ref{1} and \ref{2}, we know  that
		\begin{align*}
		&(G\Pi_{\Omega_{D}}-\Pi_{\Omega_{D}}G)f=(\overline{G}-\overline{G}+\overline{G}F_{\partial\Omega_{D}})f=\overline{G}F_{\partial\Omega_{D}}f=\overline{G}F_{\partial\Omega_{D}}f,
		\end{align*}	
	 which completes the proof.
\end{proof} 
Applying a similar argument as in Theorem \ref{th3.8} to $\overline{G}$ and $\Pi^+_{\Omega_D}$,  we have an analog of Theorem \ref{th3.8} for $\overline{G}$ as follows.
\begin{corollary}\label{co3.9}
	Let $\Omega_D \subset\mathbb{R}_*^{m+1}$ be a bounded axially symmetric domain. Suppose that $f\in {W}_{p}^{1}(\Omega_{D},d\mu)$ with $1<p<\infty$. Then
	\begin{enumerate}
\item		$\overline{G}\Pi^+_{\Omega_{D}}f=Gf,$ \label{1.1}
	\item		$\Pi^+_{\Omega_{D}}\overline{G}f=Gf-G\overline{F}_{\partial\Omega_{D}}f,$ \label{2.2}
	\item		$(\overline{G}\Pi^+_{\Omega_{D}}-\Pi^+_{\Omega_{D}}\overline{G})f=G\overline{F}_{\partial\Omega_{D}}f$.
	\end{enumerate}
\end{corollary}
Now, we can notice that the statement \ref{1} in Theorem \ref{th3.8}  and the equation \ref{1.1} in Corollary \ref{co3.9} gives rise to the following important mapping property of the  $\Pi_{\Omega_{D}}$ and $\Pi^+_{\Omega_{D}}$ operators.
\begin{proposition}
	Let $\Omega_D \subset\mathbb{R}_*^{m+1}$ be a bounded axially symmetric domain. Then, the relations 
	$$\Pi_{\Omega_{D}}:\ker\overline{G}\cap{L}^2 (\Omega_D,d\mu)\longrightarrow \ker  G\cap{L}^2 (\Omega_D,d\mu),$$
	$$\Pi^+_{\Omega_{D}}:\ker G\cap{L}^2 (\Omega_D,d\mu)\longrightarrow \ker \overline{G}\cap{L}^2 (\Omega_D,d\mu),$$
	are true.
\end{proposition}
From the statement \ref{2} in Theorem \ref{th3.8}, the statement \ref{2.2} in Corollary \ref{co3.9} and $F_{\partial\Omega_D}f=\overline{F}_{\partial\Omega_D}f=0$ for any function $f\in \mathring{{W}}_2^1 (\Omega_D,d\mu)$, we get the following 
\begin{proposition}\label{pr3.11}
		Let $\Omega_D \subset\mathbb{R}_*^{m+1}$ be a bounded axially symmetric domain. Then, the relations
		\begin{align*}
		&\Pi_{\Omega_{D}}:G(\mathring{{W}}_2^1 (\Omega_D),d\mu)\longrightarrow\overline{G}(\mathring{{W}}_2^1 (\Omega_D),d\mu),\\
		&\Pi^+_{\Omega_{D}}:\overline{G}(\mathring{{W}}_2^1 (\Omega_D),d\mu)\longrightarrow G(\mathring{{W}}_2^1 (\Omega_D),d\mu),	
		\end{align*} 
		are true.
\end{proposition}
	Then, we get the following connection between the $\Pi_{\Omega_{D}}$-operator and  $\Pi^+_{\Omega_{D}}$-operator.
	\begin{theorem}
	Let $\Omega_D \subset\mathbb{R}_*^{m+1}$ be a bounded axially symmetric domain.	 Suppose $f\in {W}_{p}^{k}(\Omega_{D},d\mu)$ with $k\in \mathbb{N}\cup\{0\},1<p<\infty$. Then we have 
		\begin{align*}
		&\Pi ^+_{\Omega_{D}}\Pi_{\Omega_{D}}f=f-G\overline{F}_{\partial\Omega_{D}}T_{\Omega_{D}}f, \\
		&\Pi_{\Omega_{D}}\Pi ^+_{\Omega_{D}}f
		=f-\overline{G}F_{\partial\Omega_{D}}\overline{T}_{\Omega_{D}}f.
		\end{align*}	
\end{theorem} 
\begin{proof}
These statements can be verified by a straightforward calculation. Indeed, we have
\begin{align*}
		&\Pi ^+_{\Omega_{D}}\Pi_{\Omega_{D}}f=G\overline{T}_{\Omega_{D}}\overline{G}T_{\Omega_{D}}f=G(\bar{I}-\overline{F}_{\partial\Omega_{D}})T_{\Omega_{D}}f\\
		=&(G\bar{I}T_{\Omega_{D}}-G\overline{F}_{\partial\Omega_{D}}T_{\Omega_{D}})f=(I-G\overline{F}_{\partial\Omega_{D}}T_{\Omega_{D}})f\\
		=&f-G\overline{F}_{\partial\Omega_{D}}T_{\Omega_{D}}f, \\
		&\Pi_{\Omega_{D}}\Pi ^+_{\Omega_{D}}f=\overline{G}T_{\Omega_{D}}G\overline{T}_{\Omega_{D}}f=\overline{G}(I-F_{\partial\Omega_{D}})\overline{T}_{\Omega_{D}}f\\
		=&(\overline{G}\overline{T}_{\Omega_{D}}-\overline{G}F_{\partial\Omega_{D}}\overline{T}_{\Omega_{D}})f=(I-\overline{G}F_{\partial\Omega_{D}}\overline{T}_{\Omega_{D}})f\\
		=&f-\overline{G}F_{\partial\Omega_{D}}\overline{T}_{\Omega_{D}}f.
		\end{align*}	
		This completes the proof.
\end{proof}
The following formula for a function $f\in{W}_{p}^{k}(\Omega_{D})(k\in \mathbb{N};1<p<\infty)$ is also of some interest for the invertibility of  $\Pi_{\Omega_{D}}$.
\begin{corollary}
Let $\Omega_D \subset\mathbb{R}_*^{m+1}$ be a bounded axially symmetric domain. Suppose $f\in{W}_{p}^{k}(\Omega_{D},d\mu)(k\in \mathbb{N};1<p<\infty)$. Then
	\begin{align*}
	\overline{T}_{\Omega_{D}}\overline{G}\Pi^+_{\Omega_{D}}\Pi_{\Omega_{D}}f=\overline{T}_{\Omega_{D}}\overline{G}f
	\end{align*}
	 holds.
\end{corollary}
\begin{proof}
Let $f\in {W}_{p}^{k}(\Omega_{D},d\mu)(k\in \mathbb{N};p>\{m,2\}).$ Then we obtain 
	\begin{align*}
 \overline{T}_{\Omega_{D}}\overline{G}\Pi^+_{\Omega_{D}}\Pi_{\Omega_{D}}f&=\overline{T}_{\Omega_{D}}\overline{G}G\overline{T}_{\Omega_{D}}\overline{G}T_{\Omega_{D}}f=\overline{T}_{\Omega_{D}}G\overline{G}\ \overline{T}_{\Omega_{D}}\overline{G}T_{\Omega_{D}}f\\
 &=\overline{T}_{\Omega_{D}}G\overline{G}T_{\Omega_{D}}f=\overline{T}_{\Omega_{D}}\overline{G}GT_{\Omega_{D}}f=\overline{T}_{\Omega_{D}}\overline{G}f ,
	\end{align*}
	where the argument above used the fact that $\overline{G}$ is a left inverse of $\overline{T}_{\Omega_D}$, which can be obtained similarly as for the fact that $G$ is a left inverse of $T_{\Omega_D}$.
\end{proof}
Let us now consider the space $C_0^{\infty}(\Omega_D)$ and the following result holds.
\begin{corollary}
Let $\Omega_D \subset\mathbb{R}_*^{m+1}$ be a bounded axially symmetric domain and $f\in C_0^{\infty}(\Omega_D)$. Then, we have
\begin{align*}
\Pi ^+_{\Omega_{D}}\Pi_{\Omega_{D}}f=f,\ \Pi_{\Omega_{D}}\Pi ^+_{\Omega_{D}}f=f.
\end{align*}
\end{corollary}
\begin{proof}
This corollary is an immediate consequence of  Proposition \ref{pr3.11}. Indeed, with the Borel-Pompeiu formula \eqref{BPF1}, we have
\begin{align*}
\overline{F}_{\partial\Omega_{D}}T_{\Omega_{D}}f=\overline{F}_{\partial\Omega_{D}}(I-{F}_{\partial\Omega_{D}})f=\overline{F}_{\partial\Omega_{D}}f-\overline{F}_{\partial\Omega_{D}}{F}_{\partial\Omega_{D}}f=0,
\end{align*}
for all $f\in C_0^{\infty}(\Omega_D)$. Hence, statement 1 of Proposition \ref{pr3.11} tells us that $\Pi ^+_{\Omega_{D}}\Pi_{\Omega_{D}}f=f$. The other equation can be verified similarly.
\end{proof}
Since $C_0^{\infty}(\Omega_D)$ is dense in $L^2(\Omega_D,d\mu)$, and $\Pi_{\Omega_D}$ is bounded acting on $L^2(\Omega_D,d\mu)$, we have
\begin{corollary}\label{inverse}
Let $\Omega_D \subset\mathbb{R}_*^{m+1}$ be a bounded axially symmetric domain and $f\in L^2(\Omega_D,d\mu)$. Then, we have
\begin{align*}
\Pi ^+_{\Omega_{D}}\Pi_{\Omega_{D}}f=f,\ \Pi_{\Omega_{D}}\Pi ^+_{\Omega_{D}}f=f.
\end{align*}
\end{corollary}
\begin{remark}
The corollary above shows us that $\Pi ^+_{\Omega_{D}}$ is a left and right inverse of $\Pi_{\Omega_{D}}$.
\end{remark}
Now, we investigate the adjoint operator of $\Pi_{\Omega_D}$, denoted by $\Pi_{\Omega_D}^*$, with respect to the weighted $Cl_{m}$-valued inner product 
$$\langle f,g\rangle_{\mu}=\int_{\Omega_{D}}\overline{f(\boldsymbol{q})}g(\boldsymbol{q})d\mu(\boldsymbol{q}),$$
where $f,g\in L^2(\Omega_D,d\mu)$. The explicit expression for $\Pi_{\Omega_D}^*$ is given as follows.
\begin{theorem}
	Let $\Omega_D \subset\mathbb{R}_*^{m+1}$ be a bounded axially symmetric domain. Assume $f,g\in\mathring{{W}}_2^k (\Omega_D,d\mu),(k\in \mathbb{N})$, then 
	\begin{align*}
		\langle\Pi_{\Omega_{D}}f,g\rangle_{\mu}=\langle f,\Pi^*_{\Omega_{D}}g\rangle_{\mu} 
			\end{align*}	and $\Pi^*_{\Omega_{D}}=T^*_{\Omega_{D}}\overline{G}^*$ holds,						where
	\begin{align*}	
			 \overline{G}^*&=-G+(m-1)\frac{\ubq}{|\ubq|^2},\ T^*_{\Omega_{D}}=T_{\Omega_D}.
	\end{align*}
\end{theorem}
\begin{proof}
	From the definitions of $\Pi_{\Omega_{D}}$-operator and $L^2$-adjoint operator, we have
\begin{align*}	\langle\Pi_{\Omega_{D}}f,g\rangle_{\mu}=\langle\overline{G}T_{\Omega_{D}}f,g\rangle_{\mu}=\langle T_{\Omega_{D}}f,\overline{G}^*g\rangle_{\mu}=\langle f,T^*_{\Omega_{D}}\overline{G}^* g\rangle_{\mu}.
\end{align*}
we know that
\begin{align*}	
	\langle T_{\Omega_{D}}f\rangle_{\mu}&=\int_{\Omega_{D}}\overline{T_{\Omega_{D}}f(\boldsymbol{q})}g(\boldsymbol{q})d\mu(\boldsymbol{q})\\
	&=-\frac{1}{2\pi }\int_{\Omega_{D}}\int_{\Omega_{D}}\overline{K(\boldsymbol{q},\boldsymbol{x})f(\boldsymbol{x})}dV(\boldsymbol{x})g(\boldsymbol{q})d\mu(\boldsymbol{q})\\
	&=-\frac{1}{2\pi }\int_{\Omega_{D}}\int_{\Omega_{D}}\overline{f}(\boldsymbol{x}) \overline{\frac{2S^{-1}(\bq,\bx)}{\omega_{m-1}|\ubx|^{m-1}}}dV(\boldsymbol{x})g(\boldsymbol{q})|\ubq|^{1-m}dV(\boldsymbol{q})\\
		&=-\frac{1}{2\pi }\int_{\Omega_{D}}\overline{f(\boldsymbol{x})}\left(\int_{\Omega_{D}}\overline{\frac{2S^{-1}(\bq,\bx)}{\omega_{m-1}|\ubq|^{m-1}}}g(\boldsymbol{q})dV(\boldsymbol{q}) \right) d\mu(\boldsymbol{x})\\
		&=-\frac{1}{2\pi }\int_{\Omega_{D}}\overline{f(\boldsymbol{x})}\left(\int_{\Omega_{D}}{\frac{2S^{-1}(\bx,\bq)}{\omega_{m-1}|\ubq|^{m-1}}}g(\boldsymbol{q})dV(\boldsymbol{q}) \right) d\mu(\boldsymbol{x})\\
		&=-\frac{1}{2\pi }\int_{\Omega_{D}}\overline{f(\boldsymbol{x})}\left(\int_{\Omega_{D}}K(\bx,\bq)g(\boldsymbol{q})dV(\boldsymbol{q}) \right) d\mu(\boldsymbol{x})\\
		&=\langle f,T_{\Omega_{D}}g\rangle_{\mu}.
\end{align*}
Hence, we have  $T^*_{\Omega_{D}}=T_{\Omega_D}$. Similarly,
\begin{align*}	
	\langle Gf,g\rangle_{\mu}&=\int_{\Omega_{D}}\overline{Gf(\boldsymbol{q})}g(\boldsymbol{q})d\mu(\boldsymbol{q})=\int_{\Omega_{D}}(\overline{f}\overline{G})(\boldsymbol{q})g(\boldsymbol{q})|\ubq|^{1-m}dV(\boldsymbol{q})\\
	&=\int_{\partial\Omega_{D}}\overline{f(\boldsymbol{q})}d\boldsymbol{q}^*g(\boldsymbol{q})|\ubq|^{1-m}-\int_{\Omega_{D}}\overline{f(\boldsymbol{q})}\bigg[\overline{G}\big(g(\boldsymbol{q})|\ubq|^{1-m}\big)\bigg]dV(\boldsymbol{q})\\
	&=-\int_{\Omega_{D}}\overline{f(\boldsymbol{q})}\bigg[\overline{G}\big(g(\boldsymbol{q})|\ubq|^{1-m}\big)\bigg]dV(\boldsymbol{q})\\
	&=-\int_{\Omega_{D}}\overline{f(\boldsymbol{q})}\bigg[\bigg(\frac{\partial}{\partial q_0}-\frac{\ubq}{|\ubq|^2}E_{\ubq}\bigg)\big(g(\boldsymbol{q})|\ubq|^{1-m}\big)\bigg]dV(\boldsymbol{q})\\
	&=-\int_{\Omega_{D}}\overline{f(\boldsymbol{q})}|\ubq|^{1-m}\bigg[\bigg(\frac{\partial}{\partial q_0}-\frac{\ubq}{|\ubq|^2}(E_{\ubq}+1-m)\bigg)g(\boldsymbol{q})\bigg]dV(\boldsymbol{q})\\
	&=-\int_{\Omega_{D}}\overline{f(\boldsymbol{q})}\bigg[\bigg(\frac{\partial}{\partial q_0}-\frac{\ubq}{|\ubq|^2}E_{\ubq}\bigg)+(m-1)\frac{\ubq}{|\ubq|^2}\bigg]g(\boldsymbol{q})d\mu(\boldsymbol{q})\\
	&=-\int_{\Omega_{D}}\overline{f(\boldsymbol{q})}\bigg[\overline{G}+(m-1)\frac{\ubq}{|\ubq|^2}\bigg]g(\boldsymbol{q})d\mu(\boldsymbol{q})\\
 &=\bigg\langle f,\bigg(-\overline{G}+(1-m)\frac{\ubq}{|\ubq|^2}\bigg)g\bigg\rangle_{\mu},
\end{align*}
which implies that $G^*=-\overline{G}+(1-m)\frac{\ubq}{|\ubq|^2}$, and we immediately have
$\overline{G}^*=-G+(m-1)\frac{\ubq}{|\ubq|^2}$, which completes the proof.
\end{proof}
\begin{remark}
In the classical case, the $\Pi$-operator has been shown to be an $L^2$-isometry, which is different from the slice case here. However, in the argument above, one can see that the Teodorescu transform $T_{\Omega_D}$ is an $L^2$ isometry with respect to the weighted inner product.
\end{remark}

\section{Existence of solutions of a slice Beltrami equation}
In \cite{G2}, the authors used a generalized $\Pi$-operator to solve a hypercomplex Beltrami equation. Similarly, in this section, we will see how the slice Beltrami equation can be solved  using   $\Pi_{\Omega_D}$-operator .
\par
Let ${\Omega_D}\subset\mathbb{R}^{m+1}_*$  be a bounded axially symmetric domain with smooth boundary $\partial_{\Omega_D}$, and $f:\Omega_D\longrightarrow Cl_m$ be a slice monogenic function. Moreover, let $\omega:\Omega_D\longrightarrow Cl_m$ be a sufficiently smooth function. Then we call the equation 
	\begin{align}	\label{eq4.1}
		G\omega=f\overline{G}\omega
	\end{align}	a slice Beltrami equation. Applying the ansatz
	\begin{align}	\label{eq4.2}
	\omega=\phi+T_{\Omega_D}h,
\end{align}		where $\phi$ is an arbitrary left-slice monogenic function. Inserting this ansatz for $\omega$ into \eqref{eq4.1} leads to 
\begin{align*}
	G(\phi+T_{\Omega_D}h)=f(\overline{G}(\phi+T_{\Omega_D}h))\iff 	G\phi+GT_{\Omega_D}h=f(\overline{G}\phi+\overline{G}T_{\Omega_D}h).
	\end{align*}
	Since $T$ is the  right inverse of the $G$, and $\phi$ is  left-slice monogenic which means that $G\phi=0$, the latter equation simplifies to the following  fixed point equation for $h$ :  
	\begin{align}	\label{eq4.3}
h=f(\overline{G}\phi+\Pi_{\Omega_{D}}h).
\end{align}	
On the one hand, $\omega$ is a solution of \eqref{eq4.1}, if $h$ is a solution of  \eqref{eq4.3}. On the other hand, each solution of \eqref{eq4.1} can be represented by \eqref{eq4.2}. 
\par
We review the Banach fixed-point theorem as follows.
\begin{theorem}[Banach Fixed-Point Theorem]
	Let $S$ be a closed subset of Banach space $X$ and let $T$ be a mapping that  $T : S\rightarrow S$. Suppose that 
	\begin{align*}
		||T(x)-T(y)||\leq\rho||x-y||, \forall x,y \in S,0\leq\rho<1.
	\end{align*}	
	 Then, there exists a unique $x^*\in S$ satisfying $x^*=T(x^*)$.
\end{theorem}
To apply the Banach Fixed-Point Theorem to \eqref{eq4.3}, we denote $$T:\ h\rightarrow f\overline{G}\phi+f\Pi_{\Omega_{D}}h,$$ in our case, we have
	\begin{align*}	
		||T(h_1)-T(h_2)||_{L^p(d\mu)}&=||f(\overline{G}\phi+\Pi_{\Omega_{D}}h_1)-f(\overline{G}\phi+\Pi_{\Omega_{D}}h_2)||_{L^p(d\mu)}\\
		&\leq||f\Pi_{\Omega_{D}}||_{L^p(d\mu)}\cdot||h_1-h_2||_{L^p(d\mu)}.
	\end{align*}
	Therefore, with the Banach Fixed-Point Theorem, when 
	\begin{align*}		
	|| f||_{L^p(d\mu)}\leq\dfrac{1}{|| \Pi_{\Omega_{D}} ||_{L^p(d\mu)}},
\end{align*}
$T$ is contractive, which leads to the existence of a unique solution to \eqref{eq4.3}. Applying our norm estimate \eqref{eq3.5} for the $\Pi_{\Omega_{D}}$-operator we get the condition 
 	\begin{align*}	
 	|| f||_{L^p(d\mu)}\leq\bigg({\frac{\omega_{m-1}\pi}{4(2\pi)^{\frac{1}{4}}C'^{\frac{1}{2}}}}\bigg)^{\frac{1}{p}}
 \end{align*}being sufficient for the existence of a solution of equation \eqref{eq4.3}.
\subsection*{Acknowledgments}
The work of Chao Ding is supported by National Natural Science Foundation of China (No. 12271001), Natural Science Foundation of Anhui Province (No. 2308085MA03) and Excellent University Research and Innovation Team in Anhui Province (No. 2024AH010002).

\subsection*{Data Availability}
No new data were created or analysed during this study. Data sharing is not applicable to this article.



\begin{thebibliography}{1}
\bibitem{Ri} R. Abreu-Blaya, J. Bory-Reyes, A. Guzm\'{a}n-Ad\'{a}n, U. K\"{a}hler,\emph{On the $\Pi$-operator in Clifford analysis}, J. Math. Anal. Appl. , 434(2016),1138-1159.

\bibitem{Bi} C. Bisi, J. Winkelmann, \emph{The harmonicity of slice regular functions}, J. Geom.  Anal., 31(2021), 7773–7811.

\bibitem{2}F. Brackx, R. Delanghe, F. Sommen, \emph{Clifford Analysis}, Pitman, London, 1982.

\bibitem{Co4} F. Colombo, J.O. Gonz\'{a}lez-Cervantes, I. Sabadini, \emph{A non constant coefficients differential operator associated to slice monogenic functions}, Trans. Am. Math. Soc., 365(2013), 303–318.	

\bibitem{Co1} F. Colombo, I. Sabadini, D. C. Struppa, \emph{ Noncommutative Functional Calculus, Theory and Applications of Slice Hyperholomorphic Functions},  Progress in Mathematics 289, Birkh\"{a}user, 2011.	

\bibitem{Co2} F. Colombo, I. Sabadini, D. C. Struppa, \emph{ An extension theorem for slice monogenic functions and some of its consequences}, Israel J. Math. , 177(2010), 369–489.

\bibitem{Co3} F. Colombo, I. Sabadini, and D. C. Struppa,\emph{ Algebraic properties of the module of slice regular functions in several quaternionic variables}, Indiana Univ. Math. J., 61(2012), 1581–1602.

\bibitem{Cu} C. G. Cullen, \emph{An integral theorem for analytic intrinsic functions on quaternions}, Duke Math. J. , 32(1965), 139–148.

\bibitem{10} R. Delanghe, F. Sommen, V. Sou\v cek, \emph{Clifford Analysis  and Spinor Valued Functions}, Kluwer Academic Dordrecht, 1992.

\bibitem{Ding}C. Ding, Z. Xu, \emph{Teodorescu transform for slice monogenic functions and applications}, http://arxiv.org/abs/2402.01997.

\bibitem{Ding2} C. Ding, X.Q. Cheng,\emph{ Integral formulas for slice Cauchy-Riemann operator and applications},  Adv. Appl. Clifford Algebras 34, 32 (2024).

\bibitem{Ge3} G. Gentili, C. Stoppato, D. Struppa, \emph{Regular Functions of a Quaternionic Variable},  Springer Berlin, Heidelberg, 2013.

\bibitem{Ge2} G. Gentili, D .C. Struppa, \emph{A new theory of regular function of a quaternionic variable}, Adv. Math. , 216(2007), 279–301.

\bibitem{Ge1} G. Gentili, D. C. Struppa, \emph{A new approach to Cullen-regular functions of a quaternionic variable}, C. R. Math. Acad. Sci. Paris. , 342(2006), 741–744.

\bibitem{Gh} R. Ghilnoi and A. Perotti, \emph{Slice regular functions on real alternative algebras}, Adv. Math. , 226(2011), 1662–1691.

\bibitem{Gh1} R. Ghilnoi and A. Perotti, \emph{Volume Cauchy formulas for slice functions on real associative$^{*}$-algebras}, Complex Variables and Elliptic Equations, 58(2013), 1701–1714.

\bibitem{22}\textsc{J. O. Gonz\'alez-Cervantes, D. Gonz\'alez-Campos}, The global Borel-Pompeiu-type formula for quaternionic slice regular functions, \textit{Complex Var. Elliptic Equ.}, \textbf{66}(2021), 721--730.

\bibitem{G} K. G\"{u}rlebeck,  K. Habetha, W. Spr\"{o}big, \emph{Holomorphic functions in the plane and $n$ dimensional space}, Birkh\"{a}user Verlag, Basel, 2008.

\bibitem{G1} K. G\"{u}rlebeck,  K. Habetha, W. Spr\"{o}big, \emph{Applications of holomprphic functions in two and higher dimensions}, Birkh\"{a}user Verlag, Basel, 2016.

\bibitem{G2} K. G\"{u}rlebeck , U. K\"{a}hler, \emph{On a Spatial Generalization of the Complex $\Pi$-operator}, Journal for Analysis and its Applications Volume 15(1996), No. 2, 283-297.

\bibitem{Mi} S. G. Michlin, S. Pr\"{o}bdorf,\emph{ Singular Integral Operators}, Akademie Verlag, Berlin, 1986.

\bibitem{Vekua}I.N. Vekua,\emph{Generalized Analytic Functions}, (in Russian), Nauka,
Moscow; English transl. Pergamon Press, Oxford 1962.
\end{thebibliography}

\end{document}